\documentclass[reqno, 11pt]{amsart}

\usepackage[foot]{amsaddr}
\usepackage[style = numeric-comp,
            maxbibnames = 100,
            maxcitenames = 2,
            isbn = false,
            giveninits = true,
            backend = biber,
            url=false]{biblatex}
\usepackage{booktabs, array} 
\usepackage{hyperref} 
\usepackage{graphicx} 
\usepackage{tikz}
\usepackage{xcolor, soul}
\sethlcolor{yellow}
\colorlet{shadecolor}{yellow}
\usepackage{mathtools, nccmath}
\usepackage{framed}
\usepackage{enumitem}
\usepackage{caption, subcaption} 
\usepackage{algorithm} 
\usepackage{algpseudocode} 
\usepackage{longtable}
\algrenewcommand\algorithmicrequire{\textbf{Input:}}
\algrenewcommand\algorithmicensure{\textbf{Output:}}


\usepackage{geometry} 
\allowdisplaybreaks 
\graphicspath{ {./images/} }
\DeclareMathOperator{\Tr}{Tr}
\newcommand{\reals}[1]{\ensuremath{\mathbb{R}^{#1}}}
\DeclareMathOperator{\rank}{rank}
    \newtheorem{theorem}{Theorem}
    \newtheorem{lemma}[theorem]{Lemma}
    
    \newtheorem{proposition}{Proposition}
    \newtheorem{corollary}{Corollary}

\newcommand{\projector}{\ensuremath{P}}
\newcommand{\projectorT}{\ensuremath{P^{\top}}}
\newcommand{\dimoriginal}{\ensuremath{n}}
\newcommand{\dimprojector}{\ensuremath{k}}
\newcommand{\numconstraints}{\ensuremath{m}}
\newcommand{\projectormap}[1]{\ensuremath{\projector{#1}\projectorT}}
\newcommand{\projectorinvmap}[1]{\ensuremath{\projectorT{#1}\projector}}
\newcommand{\objectivesdp}{\ensuremath{\langle C , X \rangle}}
\newcommand{\constraintsdp}{\ensuremath{\langle A_i, X \rangle}}
\newcommand{\ubidual}{\ensuremath{ub_i}}
\newcommand{\lbidual}{\ensuremath{lb_i}}
\newcommand{\subgausnorm}[1]{\ensuremath{\|{#1}\|_{\psi_2}}}

\newcommand{\twonorm}[1]{\ensuremath{\|{#1}\|_{2}}}
\newcommand{\norm}[1]{\ensuremath{\|{#1}\|}}
\newcommand{\fnorm}[1]{\ensuremath{\|{#1}\|_F}}

\newcommand{\expectedvalue}[1]{\ensuremath{\mathbb{E}[{#1}]}}
\newcommand{\probability}[1]{\ensuremath{\mathbb{P}[{#1}]}}
\newcommand{\bigO}{\ensuremath{\mathcal{O}}}
\newcommand{\mainconstant}{\ensuremath{\mathcal{C}}}

\author{Monse Guedes-Ayala$^1$}
\email{mguedes@ed.ac.uk (M. Guedes-Ayala)}
\address{$^1$School of Mathematics, University of Edinburgh, Scotland}
\author{Pierre-Louis Poirion$^2$}
\address{$^2$Center for Advanced Intelligence Project, RIKEN, Japan}
\author{Lars Schewe$^1$}
\author{Akiko Takeda$^{2,}$$^3$}
\address{$^3$Department of Creative Informatics, Graduate School of Information Science and Technology, University of Tokyo, Japan}
\email{pierre-louis.poirion@riken.jp (P.L. Poirion)}
\email{lars.schewe@ed.ac.uk (L. Schewe)}
\email{takeda@mist.i.u-tokyo.ac.jp (A. Takeda)}

\title[Sparse Sub-gaussian Random Projections for Semidefinite Programming Relaxations]{Sparse sub-gaussian Random Projections for Semidefinite Programming Relaxations}

\begin{document}

\begin{abstract}
Random projection, a dimensionality reduction technique, has been found useful in recent years for reducing the size of optimization problems. 
In this paper, we explore the use of sparse sub-gaussian random projections to approximate semidefinite programming (SDP) problems by reducing the size of matrix variables, thereby solving the original problem with much less computational effort.
We provide some theoretical bounds on the quality of the projection in terms of feasibility and optimality that explicitly depend on the sparsity parameter of the projector. 
We investigate the performance of the approach for semidefinite relaxations appearing in polynomial optimization, with a focus on combinatorial optimization problems.
In particular, we apply our method to the semidefinite relaxations of \textsc{maxcut} and \textsc{max-2-sat}. 
We show that for large unweighted graphs, 
we can obtain a good bound by solving a projection of the semidefinite relaxation of \textsc{maxcut}. 
We also explore how to apply our method to find the stability number of four classes of imperfect graphs by solving a projection of the second level of the Lasserre Hierarchy. 
Overall, our computational experiments show that semidefinite programming problems appearing as relaxations of combinatorial optimization problems can be approximately solved using random projections as long as the number of constraints is not too large. 
\end{abstract}

\maketitle

\section{Introduction}\label{sec: introduction}
Random projection (RP), sometimes also referred to as sketching, is a dimensionality reduction technique that has traditionally been used to speed up Euclidean distance-based algorithms such as $k$-means clustering \cite{Boutsidis2010} or $k$-nearest neighbours \cite{Indyk2007}. 
It is based on the Johnson–Lindenstrauss Lemma (JLL), which states that a set of points in a large dimensional space can be projected into a lower dimensional space while approximately preserving the distance between the points, see \cite{Woodruff2014} for more details.
However, in recent years RPs have been proven to be useful in solving mathematical programs more efficiently. 
Several papers have explored how different types of Mathematical Programming (MP) formulation might be approximately invariant (regarding feasibility and optimality) with respect to randomly projecting the input parameters. 
In \cite{Vu2018} the authors explore how to use random projections to approximately solve linear programming problems by projecting their constraints, \cite{DAmbrosio2020} covers RPs for quadratic programs, while \cite{Liberti2021} presents results for conic programming. 
The first and last papers explore how to reduce the size of the problem by projecting the constraints, while the second investigates how to reduce the dimension of the variables. 

Semidefinite programming (SDP) is the optimization of linear functions over the intersection of the positive semidefinite cone and an affine subspace. 
The most common solution approaches for general SDP problems are interior point methods (IPM). 
However, these methods run into scalability problems since they require storing large dense matrices. 
The most common solution for increasing scalability of SDP is to exploit some features of the problem such as sparsity \cite{Fukuda2001, Laserre2015, Zheng2018, Burer2003} or symmetry \cite{Gatermann2004, Vallentin2009, Klerk2010}. 
It is also possible to find a feasible solution for large problems by limiting the rank of the solution \cite{Burer2003Low, Burer2004, Journe2010}. 
Two other proposed solution approaches are the Alternating-Direction Method of Multipliers (ADMM) \cite{Donoghue2013} and augmented Lagrangian methods \cite{Zhao2010}. 
First-order algorithms have also been developed such as that of \cite{Renegar2017}, the iterative projection methods of Henrion and Malick \cite{Henrion2011}, the conditional gradient–based augmented Lagrangian framework of Yurtsever et al. \cite{Yurtsever2019}, or the regularization methods of Nie and Wang \cite{Nie2009}, which are tailored to SDP relaxations of polynomial optimization. 
More recently, RPs have been proposed to reduce the size of SDP problems. 
Liberti et al. explore in \cite{Liberti2023} how to use RPs to solve semidefinite programming problems by reducing the number of constraints of the problem.
In \cite{Poirion2023} the authors extend the framework of RPs for constraint aggregation to linear and semidefinite problems with linear inequalities. 
Articles such as \cite{Clarkson2017, Tropp2017, Ding2019} also explore random projections for low-rank approximations for semidefinite programming.
The authors of \cite{Bluhm2019} investigate how to reduce the size of the matrix appearing in SDP, instead of aggregating constraints, by applying a congruence operator with a random projection. 
This is similar to the approach presented in this paper. However, we improve the bounds proposed by \cite{Bluhm2019} and update some of the most basic results of RPs for sub-gaussian projectors to make them dependent of the sparsity parameter of the projector. 
This is key to be able to estimate in advance the performance of a sparse projector. 
Moreover, we test the performance of our method, contrary to \cite{Bluhm2019}, where only theoretical results are presented.

In this paper, we extend the current theory of RPs for SDP and propose some applications in combinatorial optimization.
We propose a new method for reducing the size of the matrices appearing in SDP using sparse sub-gaussian random projections.
We update the most basic RP results, including the JLL, to include the sparsity parameter of the sub-gaussian projector in the probabilistic bounds. 
To the best of our knowledge, this is the first paper that presents optimality and feasibility bounds for MPs that depend on the sparsity of the projector. 
We show that we can solve a projected version of semidefinite relaxation with a bounded error. 
We also prove that the projected problem is feasible with high probability, up to some feasibility tolerance. 
This document is structured as follows. Following this introduction, we present in Section \ref{sec: RP for SDP} the projected problem and some basic results. 
In Section \ref{sec: properties sub-gaussian RPs} we present some of the most important properties of sparse sub-gaussian random projections that will be used throughout the rest of the paper. 
Section \ref{sec: feasibility and optimality bound} presents the main results of the paper, which state that the feasibility and optimality of SDPs can be approximately preserved by RPs. 
We discuss the computational complexity of our methodology in Section \ref{sec: computational complexity}. 
Then, in Section \ref{sec: computational experiments} we present some computational experiments. 
In particular, we approximately solve the semidefinite relaxations of \textsc{maxcut} and \textsc{max-2-sat} and the second level of the SOS relaxation of the stable set problem. 
Lastly, we conclude by discussing possible future research directions.

\section{Projecting Semidefinite Programming Problems}\label{sec: RP for SDP}

A well-known result in geometry and machine learning is the Johnson-Lindenstrauss Lemma (JLL), which states that for a set of high dimensional points, we can randomly project the points to a lower dimensional space while approximately preserving the Euclidean distances between the points. More formally, it states that for any $\epsilon \in (0,1)$ and for any finite set $T \subseteq \reals{\dimoriginal}$ with $m$ points and $\dimprojector = O( \frac{\ln m}{\epsilon^2})$  there is a linear map $f: \reals{\dimoriginal} \rightarrow \reals{\dimprojector}$ such that 
\begin{equation}\label{basic JLL}
    (1-\epsilon) \norm{x - y}^2 \leq \norm{f(x) - f(y)}^2 \leq (1+ \epsilon)\norm{x - y}^2 \ \text{for all } x,y \in T.
\end{equation}
An even more basic result, which is widely used in \cite{Liberti2021, Dirksen2015, DAmbrosio2020, Vu2018}, is that a random projection is a matrix $\projector \in \mathbb{R}^{\dimprojector \times \dimoriginal}$ sampled from a Gaussian distribution such that for any $x \in \mathbb{R}^\dimoriginal$ and $\epsilon \in (0,1)$ there is a universal constants $\mathcal{C}$ such that 
\begin{equation}\label{basic RP}
    (1-\epsilon)\|x\|^2 \leq \|\projector x\|^2 \leq (1+\epsilon)\|x\|^2 
\end{equation}
with probability at least $1 - 2e^{-\mathcal{C}\epsilon^2 \dimprojector}$. Since this probability converges exponentially to 1, we say that the event holds with arbitrarily high probability (w.a.h.p.). 
The JLL can be seen as a corollary of this statement by applying the union bound to all the differences between distinct points in $T$. 
These two results are the foundation of the theory of random projections for MPs. 
However, dense Gaussian projectors are not computationally efficient in practice. Following this, sparser sub-gaussian projectors have been proposed in papers such as \cite{DAmbrosio2020, Achlioptas2001, Matouek2008, Ailon2009}, where it has been shown that for sub-gaussian matrices both \eqref{basic JLL} and \eqref{basic RP} hold. 
Following this, from now on we assume that $\projector$ can also be a sub-gaussian projector. 
The specific distribution of the entries of $\projector$ and how this distribution might affect the probabilistic bounds in \eqref{basic JLL} and \eqref{basic RP} will be discussed in Section \ref{sec: properties sub-gaussian RPs}. 
For now, let us explore how to use a projector $\projector \in \reals{\dimprojector \times \dimoriginal}$ to reduce the variable matrix of semidefinite programming problems. 

Recall that the general SDP problem with a matrix variable $X$ over the cone of positive semidefinite (psd) matrices, i.e. matrices with only non-negative eigenvalues, can be written as
\begin{align*}\tag{SDP}\label{sdp}
    \min \quad & \objectivesdp \\
    \text{s.t.} \quad & \langle A_i, X \rangle = b_i \quad  i= 1, \hdots, \numconstraints, \\
    & X \succeq 0,
\end{align*}
where $X, C, A_i$ are $\dimoriginal \times \dimoriginal$  matrices, $b$ is a given vector in $\mathbb{R}^m$, and $\langle M, N \rangle = \text{tr} (M^\top N)$ denotes the Frobenius inner product. 
As discussed before, the most common methods for solving general SDP problems are interior point methods (IPMs). 
However, recall that in practice these methods need to store dense matrices at every iteration, making them run into memory problems for large-scale applications. 
Following this, a reduction in the size of the variable matrix is very desirable for practical applications. 

We can use random projections to project the psd matrix $X$ of the SDP problem with the map $ \phi_{\projector}:  \reals{\dimoriginal \times \dimoriginal} \rightarrow  \reals{\dimprojector \times \dimprojector}$ that maps $ A \in \reals{\dimoriginal \times \dimoriginal}$ to $ \projectormap{A} \in \reals{\dimprojector \times \dimprojector}$, where $P \in \reals{\dimprojector \times \dimoriginal}$ is a random projection as defined in \eqref{basic RP}, or its sub-gaussian extension.
Recall that a map is said to be positive if it maps positive elements to positive elements. 
In our case, $\phi_{\projector}$ preserves the positive definiteness of the matrices, so it is a positive map.   
Consider the following SDP problem and its dual: \\
\hfill \\
\begin{minipage}[t]{0.45\textwidth}
\begin{align*}\tag{$\mathcal{P}$}\label{primal}
    \min \quad & \objectivesdp \\
    \text{s.t.} \quad & \constraintsdp  = b_i \quad  i= 1, \hdots, \numconstraints, \\
    & X \succeq 0
\end{align*}
\end{minipage}
\hfill 
\begin{minipage}[t]{0.45\textwidth}
\begin{align*}\tag{$\mathcal{D}$}\label{dual}
    \max \quad & b^{\top} y \\
    \text{s.t.} \quad & \sum_{i=1}^{\numconstraints} A_i y_i \preceq C, 
\end{align*}
\end{minipage}

We can then project these problems using the random projection $P \in \reals{\dimprojector \times \dimoriginal}$ and the positive map $\phi_{\projector}$ as \\
\begin{minipage}[t]{0.45\textwidth}
\begin{align*}\tag{$\mathcal{RP}$}\label{projected primal}
    \min \quad & \langle PCP^{\top}, Y \rangle \\
    \text{s.t.} \quad & \langle \projectormap{A_i},Y \rangle  = b_i \quad  i= 1, \hdots, \numconstraints, \\
    & Y \succeq 0
\end{align*}
\end{minipage}
\hfill 
\begin{minipage}[t]{0.45\textwidth}
\begin{align*}\tag{$\mathcal{RD}$}\label{projected dual}
    \max \quad & b^{\top} z \\
    \text{s.t.} \quad & \sum_{i=1}^{\numconstraints} \projectormap{A_i} z_i \preceq PC P^{\top}, \\
\end{align*}
\end{minipage} \\ 
where $Y$ is a new matrix of dimension $\dimprojector \times \dimprojector$. 
The authors of \cite{Bluhm2019} already proposed to reduce the size of SDP problems using these positive maps. 
They show that the difference between the inner product of projected and unprojected matrices cannot be bounded solely by the Frobenious norm and propose to bound it using the nuclear norm of the matrices instead. 
We update and extend their results by using a tighter bound taken from \cite{Poirion2023}, which combines the nuclear and the Frobenious norm of the matrices. 
More details of this will be discussed in Section \ref{sec: properties sub-gaussian RPs}. But now, let us see some basic relations between the projected and original problems. 

First, it is easy to see that a solution $\hat{Y}$ to problem \eqref{projected primal} is going to be feasible for \eqref{primal} when lifted by $X = \projectorinvmap{\hat{Y}}$. 
Similarly, any solution to the original dual is also feasible for the projected dual.
\begin{proposition}\label{prop: feasibility of lifting}
    For any feasible solution $Y$ to the projected problem \eqref{projected primal}, $\projectorinvmap{Y}$ is also feasible for the original problem \eqref{primal}.
\end{proposition}
\begin{proof}
    Let $\hat{Y}$ be a feasible solution to the projected problem. 
    Then,  $\langle \projectormap{A_i},\hat{Y} \rangle  = b_i$ holds for all $i$. 
    Define the lifting of $\hat{Y}$ into the original dimension as $X = \projectorinvmap{\hat{Y}}$. 
    Then, after some basic algebraic manipulations, one can check that $\langle A_i,\projectorinvmap{\hat{Y}} \rangle = b_i$ for all $i$ as well. 
    Secondly, since $\hat{Y}$ is psd, $X = \projectorinvmap{\hat{Y}}$ is psd as well. 
    To see this, consider $z^{\top}Xz = z^{\top}\projectorinvmap{Y}z$ for any vector $z \in \reals{\dimoriginal}$. 
    Let $z^{\top}P^{\top} = w^{\top}$ and $Pz = w$.
    Then we have that $z^{\top}Xz = z^{\top}\projectorinvmap{\hat{Y}}z = w^{\top} \hat{Y} w \geq 0$ for all $w \in \reals{\dimprojector}$ since $\hat{Y}$ is psd. 
\end{proof}
\begin{proposition}\label{prop: feasibility of dual}
    For any feasible dual solution $y$ of the original problem \eqref{dual}, $y$ is also feasible for the projected dual \eqref{projected dual}.
\end{proposition}
\begin{proof}
    Since $\hat{y}$ is feasible for the original dual, it satisfies $\sum_{} A_i \hat{y}_i \preceq C$. Now, since $ \phi_{\projector}: A \rightarrow \projectormap{A}$ is a positive map, $\sum_{} \phi_{\projector}(A_i) \hat{y}_i \preceq \phi_{\projector}(C)$ is also satisfied, i.e. $\hat{y}$ is also feasible for the projected dual. 
\end{proof}
This allows us to get a certificate of feasibility by solving the projected problem. 
If the projected problem is feasible, since its feasible region is contained in that of the original problem, then the original problem is also feasible. This inclusion of feasible regions can be summarised in the following corollary, which follows directly from Proposition \ref{prop: feasibility of lifting} and Proposition \ref{prop: feasibility of dual}.
\begin{corollary}\label{cor: inclusion of feasible regions}
    Let $F$ denote the feasible region of a given problem. Then the following inclusions hold:
    \begin{align*}
        F_{\eqref{projected primal}} & \subseteq F_{\eqref{primal}},\\
        F_{\eqref{dual}} & \subseteq F_{\eqref{projected dual}}.
    \end{align*}
\end{corollary}
However, Corollary \ref{cor: inclusion of feasible regions} does not tell us anything about the feasibility of the projected problem nor the boundness of the projected dual. 
The projected problem may be infeasible while the original problem is feasible. 
However, as we will see in Section \ref{sec: feasibility and optimality bound}, we can bound the $\epsilon$-feasibility of the projected problem with respect to the original problem. 
In other words, we can bound the probability of the two feasible regions being close to each other, where closeness is defined in terms of some tolerance $\epsilon$. 
Moreover, assuming that the bounds of the dual variables are known, problem \eqref{primal} can be reformulated to ensure the feasibility of the projection. 
But before presenting the main feasibility and optimality results, we need to introduce some of the main properties of sparse sub-gaussian random projections.

\section{Properties of sparse sub-gaussian random projections} \label{sec: properties sub-gaussian RPs}

In this section, we provide some background on some properties of sparse sub-gaussian random projections that will be used throughout the rest of the paper. 
Even though all the properties presented here also hold for the classic Gaussian random projections and have been widely discussed in papers such as \cite{DAmbrosio2020, Vu2018, Poirion2023, Liberti2021, Liberti2023}, we update all results so that the sparsity of the new sparse sub-gaussian random-projector appears in all the bounds, including the feasibility and optimality bounds. 
The main motivation behind this is that our computational experiments have shown that for practical applications the sparsity of the projector highly influences both the feasibility of the projected problem and the quality of the projection. 

There are many examples of sparse sub-gaussian projectors that can be used to project MPs and ML problems. For example, \cite{Ailon2009, Matouek2008} propose to sample the entries of $\projector$ independently where each entry is 0 with probability $1-\gamma$, and either $\sqrt{\gamma}$ or $-\sqrt{\gamma}$ with probability $\frac{\gamma}{2}$ each. 
The authors of \cite{DAmbrosio2020} propose a new sparse sub-gaussian projector $\projector \in \reals{\dimprojector \times \dimoriginal}$, where each entry of $\projector$ is sampled from a normal distribution $N(0,\frac{1}{\sqrt{\gamma \dimprojector}})$ with probability $\gamma$, and is $0$ with probability $1-\gamma$. 
In \cite{DAmbrosio2020} the authors show that this projector is indeed a random projection since all the rows are sub-gaussian isotropic with zero mean. 
However, the optimality and feasibility bounds presented in \cite{DAmbrosio2020} are not updated to take the sparsity parameter into account. 
In the rest of this section, we focus on this last sparse projector and propose new probabilistic bounds that depend on the sparsity parameter $\gamma$. 
To the best of our knowledge, this is the first paper that explicitly includes the sparsity of the projector in the probabilistic bounds. Let us start by updating Lemma 3.1 in \cite{Vu2019ball}.

\begin{lemma} \label{lemma: sub-gaussian spectral bound}
    Let $\projector \in \reals{\dimprojector \times \dimoriginal}$ be a sparse sub-gaussian matrix in which each component is sampled from a normal distribution N$(0, \frac{1}{\sqrt{\dimprojector \gamma}})$ with probability $\gamma \in (0,1)$ and is $0$ with probability $1-\gamma$. For $\epsilon \in (0,1)$ we have that
    \begin{equation*}
        \|\frac{1}{\dimprojector} \projector \projectorT - I_{\dimprojector}\|_2 \leq \epsilon
    \end{equation*}
    with probability at least $1-2e^{-(\frac{\sqrt{\epsilon} \gamma \dimprojector}{\mainconstant} - \sqrt{k})^2}$, 
    where $\| \cdot \|_2$ denotes the spectral norm and $\mainconstant$ is some universal constant.
\end{lemma} 
\begin{proof}
By Proposition 3 in \cite{DAmbrosio2020} we know that $\projector$ has independent, mean zero, sub-gaussian isotropic rows. Then, we can apply Theorem 4.6.1 in \cite{Vershynin2018}, which states that for any $t\geq0$ 
\begin{equation} \label{equation sub-gaussian HDP}
    \twonorm{\frac{1}{m} \projector \projectorT - I_{\dimprojector}} \leq K^2\max(\delta, \delta^2) 
\end{equation}
with probability at least $1-2e^{-t^2}$ where $K= \max_i \subgausnorm{\projector_i}$, $\subgausnorm{{}\cdot{}}$ denotes the vector sub-gaussian norm for row $i$, and $\delta = C_0 \left(\sqrt{\frac{\dimprojector}{\dimoriginal}} + \frac{t}{\sqrt{\dimoriginal}}\right)$ for some constant $C_0$. Our first goal is to bound $K$. 
Recall that all entries of $P$ are independently drawn from a probability distribution obtained by multiplying a Bernoulli random variable (RV) and a normal RV, i.e. $ P_{ij} \sim B(\gamma)N(0, \frac{1}{\sqrt{\dimprojector\gamma}})$ with Bernoulli parameter $\gamma$ and standard deviation $\frac{1}{\sqrt{\dimprojector\gamma}}$. 
Since all row vectors are drawn independently from the same distribution, we have that $K=\subgausnorm{\projector_i}$ for any row $\projector_i$ of $\projector$. Following this, and by the definition of the sub-gaussian norm of a random vector, we have that
\begin{align}
    K=\|\projector_i \|_{\psi_2} &= \sup_{x\in S^{n-1}} \subgausnorm{\langle \projector_{i}, x \rangle} \\
    & \leq \max_j C_1 \subgausnorm{\projector_{ij}} \label{spectral constant} \\
    & = C_1 \subgausnorm{\projector_{ik}} \text{ for any element }k \text{ of }P_i   \label{spectral single subnorm}
\end{align}
where $S^{n}$ denotes the unit $n$-sphere and $C_1$ is a universal constant. Line \eqref{spectral constant} follows from Lemma 3.4.2 in \cite{Vershynin2018}, while line \eqref{spectral single subnorm} follows again from the fact that all the entries of $\projector$ are independently sampled from the same distribution. 
Therefore, we only need to bound the sub-gaussian norm of any of the entries of $P$. 
To do this, we prove that the sub-gaussian norm of the multiplication of a Bernoulli distribution with parameter $\gamma$ and a normal distribution with mean 0 and standard deviation $\sigma$ is bounded above by the sub-gaussian norm of the normal distribution, which is known to be bounded by $C_2\sigma$ for some constant $C_2$. 
For this, consider a random variable $Z=XY$ where $X, Y$ are distributed as $X \sim B(\gamma)$ with bernoulli parameter $\gamma$, and $Y \sim N(0,\sigma^2)$ with variance $\sigma^2$. 
Now, we can apply the law of total expectation and get that for a function $g(Z)$:
\begin{align}
        \expectedvalue{g(Z)} & = \expectedvalue{g(Z) | X=0} \probability{X=0} + \expectedvalue{g(Z) | X=1} \probability{X=1} \\
        & = (1 - \gamma) \expectedvalue{g(Z) | X=0} + \gamma \expectedvalue{g(Z) | X=1} \\
        & = (1 - \gamma) g(0) + \gamma \expectedvalue{g(Y)} \label{spectral: combined-normal}
\end{align}
where line \eqref{spectral: combined-normal} follows from the fact that when $X=0$, then $Z=0$ as well and the expected value is just the function evaluated at 0. Now, when $X=1$, $Z$ is equivalent to being drawn from the normal distribution, so the expectation of $Z$ is equal to the expectation of $Y$. 
Then, by the definition of sub-gaussian norm, we have that:
\begin{align}
    \subgausnorm{\projector_{ij}} & = \inf \left\{t>0 : \expectedvalue{\text{exp}(Z^2/t^2)} \leq 2\right\} \\
    & = \inf \left\{ t>0 : (1 - \gamma) \text{exp}(0^2/t^2) + \gamma \expectedvalue{\text{exp}(Y^2/t^2)} \leq 2 \right\} \label{conditional sub-gaussian norm}\\
    & = \inf \left\{ t>0 : \expectedvalue{\text{exp}(Y^2/t^2)} \leq \frac{1 + \gamma}{\gamma} \right\} \label{spectral new rhs} \\
    & \leq C_2 \sigma \label{spectral bounded normal}
\end{align}
where line \eqref{conditional sub-gaussian norm} follows from applying \eqref{spectral: combined-normal} to $g(Z) = e^{\frac{Z^2}{t^2}}$, line \eqref{spectral bounded normal} follows from the fact that, as $\gamma \in (0,1)$, the right hand side of \eqref{spectral new rhs} is in the interval $(2, \infty)$, 
making the norm of $Z$ at most as large as that of $Y$.
Since the sub-gaussian norm of a normal distribution is bounded by $C_2\sigma$ for some universal constant $C_2$ and variance $\sigma$, we obtain the bound of line \eqref{spectral bounded normal}. 
Following this, substituting in the standard deviation of the normal distribution of the entries of $\projector$, we have that $K \leq C_3 \frac{1}{\dimprojector \gamma} \leq C_3 \frac{\sqrt{n}}{\dimprojector \gamma}$. 

We also need to analyse the max operator involving $\delta$ in \eqref{equation sub-gaussian HDP}. 
Consider the term $\delta = C_0 \left(\sqrt{\frac{\dimprojector}{\dimoriginal}} + \frac{t}{\sqrt{\dimoriginal}}\right)$. From the proof of Theorem 4.6.1 in \cite{Vershynin2018} we know that the constant $C_0$ in $\delta$ is the constant appearing in the sub-exponential norm of the sub-exponential random variable $\langle \projector_i, x \rangle^2 - 1$, which is bounded below by $C_0\geq 1$ (this bound follows from the proof of the centering of sub-exponential random variables). 
For $C_0$ large enough, $C_0$ dominates the expression $\delta = C_0 \left(\sqrt{\frac{\dimprojector}{\dimoriginal}} + \frac{t}{\sqrt{\dimoriginal}}\right)$ where $0 \leq \left(\sqrt{\frac{\dimprojector}{\dimoriginal}} + \frac{t}{\sqrt{\dimoriginal}}\right) \leq 1$, which implies that $\max(\delta, \delta^2) = \delta^2$. Following this, we have that with probability $1-2e^{-t^2}$
\begin{align}
    \|\frac{1}{m}AA^{\top} - I_{\dimprojector} \| & \leq K^2\max(\delta, \delta^2) \\
    & = K^2 C_0^2 \left(\sqrt{\frac{\dimprojector}{\dimoriginal}} + \frac{t}{\sqrt{\dimoriginal}}\right)^2 \\
     & \leq \frac{C_3^2 C_0^2 n}{\dimprojector^2 \gamma^2} \left(\sqrt{\frac{\dimprojector}{\dimoriginal}} + \frac{t}{\sqrt{\dimoriginal}}\right)^2 \\
    & = \frac{\mainconstant (\sqrt{\dimprojector} + t )^2}{\dimprojector^2 \gamma^2}
\end{align}
for some universal constant $\mainconstant$. This means that with probability at least $1-2e^{-(\frac{\sqrt{\epsilon} \gamma \dimprojector}{\mainconstant} - \sqrt{k})^2}$ we have
\begin{align}
    \twonorm{\frac{1}{\dimprojector}\projector \projectorT - I_{\dimprojector}} & \leq \epsilon
\end{align}
as desired. 
\end{proof}
Following this, we can redefine the classic Johnson-Lindestrauss Lemma for sub-gaussian projections taking into account the sparsity parameter. 
\begin{lemma}[JLL for Sparse Sub-Gaussian Projections]\label{lemma: jll sub-gaussian}
Let $Q \in \reals{\dimprojector \times \dimoriginal}$ be a sparse sub-gaussian matrix in which each component is sampled from a normal distribution N$(0, \frac{1}{\sqrt{\dimprojector \gamma}})$ with probability $\gamma \in (0,1)$ and is $0$ with probability $1-\gamma$ and let $\projector=\frac{1}{\sqrt{\dimprojector}} Q$. For any $\epsilon \in (0,1)$ and for any finite set $T \subseteq \reals{\dimoriginal}$ with $m$ points
\begin{equation}
    (1-\epsilon)\|x - y\|_2 \leq \|\projector x - \projector y\|_2 \leq (1+\epsilon)\|x - y\|_2
\end{equation}
with probability at least $1-2e^{-(\mainconstant\epsilon \sqrt{\dimprojector^5} \gamma^2 - \sqrt{\log m})^2}$ for some universal constant $\mathcal{C}$.
\end{lemma}

\begin{proof}
    The proof is similar to that of the sub-gaussian version of the JLL from \cite{Vershynin2018} section 9.3.1. Let $\mathcal{X}$ be a set of points in $\reals{\dimoriginal}$ and define $T$ to be the set of normalised differences between elements of $\mathcal{X}$:
    \begin{equation} \label{normalised differences}
        T = \left\{\frac{x-y}{\twonorm{x-y}} : x,y \in \mathcal{X}, x\neq y \right\}
    \end{equation}
    Recall from \cite{DAmbrosio2020} that the rows of $Q$ are independent, isotropic and sub-gaussian random vectors in \reals{\dimoriginal}. 
    By the matrix deviation inequality (Theorem 9.1.1 in \cite{Vershynin2018}), for any subset $T \in \reals{\dimoriginal}$
    \begin{equation} \label{detailed expectation JLL proof}
        \expectedvalue{\sup_{x \in T} \left| \twonorm{Q x} - \sqrt{\dimprojector} \twonorm{x} \right| } \leq C_0K^2 \gamma(T)
    \end{equation}
    where $K= \max_i \subgausnorm{\projector_i}$, $C_0$ is a universal constant, and $\gamma(T)$ denotes the gaussian complexity of $T$. Now, by Exercise 9.1.8 in \cite{Vershynin2018}, we also know that for any $u \geq 0$, the event 
    \begin{equation} \label{detailed probability JLL proof}
        \sup_{x \in T} \left| \twonorm{Q x } - \sqrt{\dimprojector}\twonorm{x} \right| \leq C_0 K^2 [w(T) + u \cdot \text{rad}(T)]
    \end{equation}
    holds with probability at least $1-2e^{-u^2}$, where $\text{rad}(T)$ is the radius of $T$, i.e. $\text{rad}(T)= \sup_{x \in T} \twonorm{x}$. 
    From the proof of Lemma \ref{lemma: sub-gaussian spectral bound} we know that $K$ is bounded by $K \leq C_1 \frac{1}{\dimprojector \gamma}$ for some constant $C_1$.
    Now, applying equation \eqref{detailed probability JLL proof} to our set $T$ in \eqref{normalised differences}, together with the bound on $K$, we have that 
    \begin{align}
        \sup_{x \in T} \left| \frac{\twonorm{Q x - Q y}}{\twonorm{x - y}} - \sqrt{\dimprojector} \right| & \leq C_2 \frac{1}{\dimprojector^2 \gamma^2} [w(T) + u \cdot \text{rad}(T)]
    \end{align}
    Multiplying both sides by $\frac{1}{\sqrt{\dimprojector}}\twonorm{x - y}$ and rearranging we get that, for a matrix $\projector = \frac{1}{\sqrt{\dimprojector}} Q$, with probability at least $1 - 2e^{-u^2}$:
    \begin{align}
        \left| \twonorm{\frac{1}{\sqrt{\dimprojector}} Q x - \frac{1}{\sqrt{\dimprojector}} Q y } - \twonorm{x - y} \right| & \leq \frac{C_2 \sqrt{\log m} + C_0u}{\gamma^2 \sqrt{\dimprojector^5}} \twonorm{x -y }
    \end{align}
    By letting $\epsilon \leq \frac{C_2 \sqrt{\log m} + C_0u}{\gamma^2 \sqrt{\dimprojector^5}}$, we have that $\projector$ is an approximate isometry on $\mathcal{X}$ , i.e.
    \begin{equation}
        (1-\epsilon) \twonorm{x - y} \leq \twonorm{\projector x - \projector y} \leq (1+\epsilon) \twonorm{x - y} 
    \end{equation}
    with probability at least $1-2e^{-(\mainconstant \epsilon \sqrt{\dimprojector^5} \gamma^2 - \sqrt{\log m})^2}$ for some constant $\mathcal{C}$, as desired. 
\end{proof}

We can also update the result in \eqref{basic RP} for the sparse sub-gaussian projector. 

\begin{lemma}\label{lemma: squared result}
Let $\projector \in \reals{\dimprojector \times \dimoriginal}$ be a sparse sub-gaussian matrix in which each component is sampled from a normal distribution N$(0, \frac{1}{\sqrt{\dimprojector \gamma}})$ with probability $\gamma \in (0,1)$ and is $0$ with probability $1-\gamma$ and let $Q=\frac{1}{\sqrt{\dimprojector}} \projector$. For any $x \in \mathbb{R}^\dimoriginal$ and $\epsilon \in (0,1)$ there is a universal constant $\mainconstant$ such that
\begin{equation}\label{basic result}
    (1-\epsilon)\|x\|_2^2 \leq \|\projector x\|_2^2 \leq (1+\epsilon)\|x\|_2^2
\end{equation}
with probability at least  $1-2e^{-\mainconstant \epsilon^2 k^5 \gamma^4}$.
\end{lemma}
\begin{proof}
    By Exercise 9.1.8 in \cite{Vershynin2018}, we know that for any $u \geq 0$ the event
    \begin{equation}
        \sup_{x \in T} \left| \twonorm{\projector x } - \sqrt{\dimprojector}\twonorm{x} \right| \leq C_0 K^2 [w(T) + u \cdot \text{rad}(T)]
    \end{equation}
    holds with probability at least $1-2e^{-u^2}$, where $\text{rad}(T)$ is the radius of $T$. 
    Now, if $T$ has only one vector $x$, then we have that  
    \begin{align}
        \left| \twonorm{\projector x } - \sqrt{\dimprojector}\twonorm{x} \right| & \leq C_0 K^2 [w(T) + u \cdot \twonorm{x}] \\
        & \leq C_0 K^2 u \twonorm{x} \\
        &  \leq \frac{C_0 u}{\dimprojector^2 \gamma^2}  \twonorm{x}
    \end{align}
    since the gaussian width of $T$ is bounded above by the gaussian complexity, which is bounded above by $C\sqrt{\log 1} = 0$. 
    By letting $Q = \frac{1}{\sqrt{\dimprojector}} \projector$ and multiplying both sides by $\frac{1}{\sqrt{\dimprojector}}$, we have that 
    \begin{equation}
        \left| \twonorm{Q x } - \twonorm{x} \right|  \leq \frac{C_0u}{\gamma^2 \sqrt{\dimprojector^5}}  \twonorm{x}.
    \end{equation}
    Following this, we have that for all $x \in \reals{\dimoriginal}$ and $\epsilon \in (0,1)$
    \begin{equation}
        (1-\epsilon) \twonorm{x} \leq \twonorm{Q x } \leq (1+\epsilon)  \twonorm{x}.
    \end{equation}
    with probability at least $1-2e^{-\frac{\epsilon^2 k^5 \gamma^4}{C_0}}$. Now, squaring both sides we get that 
    \begin{equation}
        (1-\epsilon)^2 \twonorm{x}^2 \leq \twonorm{Q x }^2 \leq (1+\epsilon)^2  \twonorm{x}^2.
    \end{equation}
    with probability at least $1-2e^{-\frac{\epsilon^2 k^5 \gamma^4}{C_0}}$ and since $1 + 2\epsilon + \epsilon^2 \leq 1 + 3\epsilon$ and $1 - 2\epsilon + \epsilon^2 \geq 1 - 3\epsilon$, we have that 
    \begin{equation}
        (1-\epsilon)  \twonorm{x}^2 \leq (1-\epsilon/3)^2 \twonorm{x}^2 \leq \twonorm{Q x }^2 \leq (1+\epsilon / 3)^2  \twonorm{x}^2 \leq \twonorm{Q x }^2 \leq (1+\epsilon)  \twonorm{x}^2.
    \end{equation}
    with probability at least $1-2e^{-\frac{\epsilon^2 k^5 \gamma^4}{C_0}} = 1-2e^{-\mainconstant \epsilon^2 k^5 \gamma^4}$, as desired.
\end{proof}
We can now use these results to update some of the other peoperties of RPs, such as the invariance of inner products and quadratic forms.

\begin{lemma}\label{lemma: sub-gaussian vector inner product}
    Let $\projector \in \reals{\dimprojector \times \dimoriginal}$ be a sparse sub-gaussian random projection such as that of Lemma \ref{lemma: squared result} and let $\epsilon \in (0,1)$. Then there is a universal constant $\mainconstant$ such that, for any vectors $x,y \in \reals{\dimoriginal}$
    \begin{equation*}
        \langle x, y \rangle - \epsilon \|x\|_2\|y\|_2 \leq \langle \projector x, \projector y \rangle \leq \langle x, y \rangle + \epsilon \|x\|_2\|y\|_2
    \end{equation*}
with probability at least $1-4e^{-\mainconstant \epsilon^2 k^5 \gamma^4}$. 
\end{lemma}
\begin{proof}
    The proof is similar to that of Proposition 1 in \cite{Vu2018} and Lemma 1 in \cite{DAmbrosio2020}. We apply Lemma \ref{basic result} to any two vectors $u + v$, $u - v$  for $u, v$ being the normalised vectors for $x,y$, and then apply the union bound.
\end{proof}

\begin{lemma}\label{lemma: sub-gaussian quadratic form}
    Let $\projector \in \reals{\dimprojector \times \dimoriginal}$ be a sparse sub-gaussian random projection such as that of Lemma \ref{lemma: squared result} and let $\epsilon \in (0,1)$. Then there is a universal constant $\mainconstant$ such that, for any vectors $x,y \in \reals{\dimoriginal}$ and squared matrix $Q$
    \begin{equation*}
        x^{}\top Qy - 3 \epsilon \|x\|_2\|y\|_2 \fnorm{Q} \leq x^{\top}\projectorT \projector Q \projectorT \projector y \leq x^{}\top Qy + 3 \epsilon \|x\|_2\|y\|_2 \fnorm{Q}
    \end{equation*}
with probability at least $1-8re^{-\mainconstant \epsilon^2 k^5 \gamma^4}$, where $r$ is the rank of $Q$. 
\end{lemma}
\begin{proof}
    The proof is similar to that of \cite{DAmbrosio2020}, but applying Lemma \ref{lemma: sub-gaussian vector inner product}.
\end{proof}

Now, can we find similar bounds for the inner product of two matrices after applying the positive transformation $\phi_{\projector}: A \rightarrow \projectormap{A}$? 
In general, we cannot bound $|\langle \projectormap{A}, \projectormap{B} \rangle - \langle A, B \rangle|$ by just the Frobenious norm, as the following result from \cite{Bluhm2019} states:
\begin{theorem}\label{thm: negative result frobenious norm}
    Let $\projector \in \reals{\dimprojector \times \dimoriginal}$ be a random projection and $A \rightarrow \projectormap{A}$ be a random positive map such that with strictly positive probability for any $A_1, \hdots, A_m \in \reals{\dimoriginal \times \dimoriginal}$ and $0 < \epsilon < \frac{1}{4}$ we have
    \begin{equation*}
        |\langle \projectormap{A_i}, \projectormap{A_j} \rangle - \langle A_i, A_j \rangle| \leq \epsilon \|A_i\|_F\|A_j\|_F
    \end{equation*}
    where $\dimoriginal$ is the original dimension and $\dimprojector$ is there projected dimension. Then $\dimprojector = \Omega(\dimprojector)$, where $\Omega$ is used to indicate an asymptotic lower bound.
\end{theorem}

The proof of Theorem \ref{thm: negative result frobenious norm} can be found in \cite{Bluhm2019}. 
This result states that in general, it is not possible to project the original matrix in a nontrivial way while having a positive probability of an error bounded by the Frobenious norm. 
However, it is possible to bound  $|\langle \projectormap{A}, \projectormap{B} \rangle - \langle A, B \rangle|$ with the nuclear norm, and with a combination of the nuclear norm and the Frobenious norm. From now on, we denote the nuclear norm by $\| \cdot \|_*$, and recall that the nuclear norm of a matrix is the sum of its singular values.  
Recall also that the Frobenious norm is always bounded above by the nuclear norm, since $\| A \|_F = \sqrt{\sum \sigma_i^2} \leq \sum \sigma_i = \| A \|_*$, where $\sigma_i$ are the singular values of the matrix $A$. 
Since results from \cite{Poirion2023} give us a tighter bound than those of \cite{Bluhm2019}, we focus on bounding the inner product with a combination of both norms from now on. 
Moreover, to make this document as self-contained as possible, we include here the proof of Lemma \ref{lemma: frobenious and nuclear bound}, which is taken from paper \cite{Poirion2023} but needs to be updated to take the sparsity of the sub-gaussian projector into account.
\begin{lemma}\label{lemma: frobenious and nuclear bound}
    Let $\projector$ be a sparse sub-gaussian random projector such as that of Lemma \ref{lemma: jll sub-gaussian} and let $\epsilon \in (0,1)$. Then for any $A, B \in \reals{\dimoriginal \times \dimoriginal}$, we have that with probability at least  $1-8 r_A r_B e^{-\mainconstant \epsilon^2 k^5 \gamma^4}$, 
    \begin{equation*}
        |\langle A, B \rangle - \langle \projectormap{A}, \projectormap{B} \rangle| \leq 3 \epsilon \|A\|_F\|B\|_*,
    \end{equation*}
    where $r_A, r_B$ are the ranks of $A$ and $B$. 
\end{lemma}

\begin{proof}
    Recall from \cite{DAmbrosio2020} that random projections approximately preserve quadratic forms. This means that for any two vectors $x,y \in \reals{\dimoriginal}$, a random projection $P \in \reals{\dimprojector \times \dimoriginal}$, and a square matrix $Q \in \reals{\dimoriginal \times \dimoriginal}$ of rank $r$, with probability at least
    $1-8re^{-\mainconstant \epsilon^2 k^5 \gamma^4}$ we have 
    \begin{equation}\label{quadratic form}
        x^{\top} Q y - 3 \| x \| \| y \| \| Q \|_F 
        \leq x^{\top} P^{\top} \projectormap{Q} P y 
        \leq  x^{\top} Q y + 3 \| x \| \| y \| \| Q \|_F.
    \end{equation}
    Now, let $A$ be a $\dimoriginal \times \dimoriginal$ matrix of rank $r_1$, and $B$ be a $\dimoriginal \times \dimoriginal$ matrix of rank $r_2$.
    Assume first that $r_2 = 1$. Then there exists scalar $\sigma > 0$ and unit vectors $u,v \in \reals{\dimoriginal}$ such that $B=\sigma u v^{\top}$. Now, by applying \eqref{quadratic form} to $Q=A$, $x=u$, and $y=v$ we have that
    \begin{align}
        \langle \projectormap{B}, \projectormap{A} \rangle & = 
        \sigma \langle \projector u v^{\top} \projectorT, \projectormap{A} \rangle \\
        & = \sigma u^{\top} \projectorT \projectormap{A} \projector v \\
        & \leq \sigma u^{\top} A v - 3 \sigma \| u \| \| v \| \| A \|_F \\
        & = \langle B, A \rangle - 3 \sigma \| A \|_F \label{projection bound}
    \end{align}
    with probability at least $1-8re^{-\mainconstant \epsilon^2 k^5 \gamma^4}$.
    We can do the same to bound $\langle \projectormap{B}, \projectormap{A} \rangle$ from the other side. 
    Consider now the case where $r_2 > 1$. We can then write $B$ using its singular value decomposition $B=U\Sigma V^{\top}$ as 
    \begin{equation*}
        B = \sum_i \sigma_i u_i v_i^{\top}
    \end{equation*}
    where $\sigma_i > 0$ are the singular values of $B$, and $u_i, v_i$ are unit vectors. 
    If we apply \eqref{projection bound} to all $u_i, v_i$, the union bound with $i \in \{1, \hdots, r_2\}$, and the definition of nuclear norm, we get the desired result.  
\end{proof}
\section{Feasibility and Optimality Bounds of The Projected Problem}\label{sec: feasibility and optimality bound}
In this section we present the main feasibility and optimality results of the paper. 
It is possible that even though the original problem is feasible, the projected one is infeasible. 
However, we can bound this infeasibility and obtain some probabilistic results about approximate feasibility. 
We say a problem is $\epsilon$-feasible if all constraints are satisfied with a tolerance of $\epsilon$. For example, the constraint $\langle A, X \rangle = b$ is $\epsilon$-feasible if there is an $X$ such that $b - \epsilon \leq \langle A, X \rangle \leq b + \epsilon$. Following this, we show that random projections can approximately preserve feasibility up to some tolerance. 

\begin{theorem}[Approximate Feasibility]\label{lemma: approx feasibility}
    Assume that the original problem \eqref{primal} is feasible and that $\Tr(X) \leq w^2$ for all feasible solutions. Let $\epsilon \geq 3(\numconstraints + 1)\delta w^2\max(\|A_i\|_F)$, then the projected problem \eqref{projected primal} is $\epsilon$-feasible with probability at least $1-16\sum_i^{m} r_{A_i}e^{-\mainconstant \delta^2 k^5 \gamma^4}$ where $r_{A_i}$ is the rank of $A_i$. 
\end{theorem}
\begin{proof}
    Semidefinite programming feasibility is equivalent to testing membership to the convex hull relaxation of the quadratic forms obtained from the data matrices of the constraints \cite{Kalantari2019}. In particular, a SDP such as \eqref{primal} with the additional assumption of $\Tr(X)\leq w^2$ is feasible if and only if $\mathbf{b}$ is in 
    \begin{equation}
        CHR = \{\sum_{i=1}^{t}\alpha_i Q(\mathbf{x_i}): \sum_{i=1}^{t}\alpha_i = 1, \alpha_i \geq 0, \|\mathbf{x_i}\| \leq w, \mathbf{x_i} \in \reals{n}\}
    \end{equation}
    for some $t$, where $Q(\mathbf{x_i}) = (\mathbf{x_i}^{\top}A_1\mathbf{x_i}, \dots, \mathbf{x_i}^{\top}A_m\mathbf{x_i})^{\top}$. By Carathéodory's theorem, if a point lies in the convex hull of a $n$-dimensional set, then it can be written as a convex combination of at most $n+1$ points. 
    Since $Q(\mathbf{x_i})$ is $\numconstraints$ dimensional, we can bound $t$ to be at most $m+1$. Following this, we say that $b$ is in $C$ if there exists pairs $(\alpha_1, \mathbf{x_1}), \dots (\alpha_{m+1}, \mathbf{x_{m+1}}) \in [0,1] \times \reals{\dimoriginal}$ such that $\sum^{\numconstraints + 1}_{i=1} \alpha_i = 1$ and $\sum_{i=1}^{\numconstraints + 1} \alpha_i Q(\mathbf{x_i}) = b$. 
    
    Testing the feasibility of the projected problem is equivalent to testing if $\mathbf{b}$ is in 
    \begin{equation}
        CHR' = \{\sum_{i=1}^{\numconstraints +1}\alpha_i Q'(\mathbf{y_i}): \sum_{i=1}^{m+1}\alpha_i = 1, \alpha_i \geq 0, \|\mathbf{y_i}\| \leq w, \mathbf{y_i} \in \reals{\dimprojector}\}
    \end{equation}
    where $Q'(\mathbf{y_i}) = (\mathbf{y_i}^{\top}PA_1P^{\top}\mathbf{y_i}, \dots, \mathbf{y_i}^{\top}PA_mP^{\top}\mathbf{y_i})^{\top}$. 
    To show that the projected problem is $\epsilon$-feasible for $\epsilon \geq 12\delta r^2 \max(\|A_i\|_F)$, consider the distance between $\mathbf{b}$ and the the projection of $C$ for arbitrary $(\alpha_1, \mathbf{x_1}), \dots (\alpha_{m+1}, \mathbf{x_{m+1}}) \in [0,1] \times \reals{\dimoriginal}$ such that $\sum^{\numconstraints + 1}_{i=1} \alpha_i = 1$ and $\sum_{i=1}^{\numconstraints + 1} \alpha_i Q(\mathbf{x_i}) = b$. This distance vector is then
    \begin{equation}
        \mathbf{d} = \left| \ \mathbf{b}  - \left( \sum_{i=1}^{\numconstraints + 1}\alpha_i \begin{bmatrix}
            \mathbf{x_i}^{\top} P^{\top} P A_1 P^{\top} P \mathbf{x_i} \\
           \vdots \\
           \mathbf{x_i}^{\top} P^{\top} P A_m P^{\top} P \mathbf{x_i}
         \end{bmatrix} 
         \right)
         \right|.
    \end{equation}
    By the approximate preservation of quadratic forms presented in the proof of Lemma \ref{lemma: frobenious and nuclear bound}, we know that for a random projection $\projector \in \reals{\dimprojector \times \dimoriginal}$, $| \mathbf{x}^{\top}A\mathbf{x} - \mathbf{x} P^{\top} \projectormap{A} P \mathbf{x}| \leq 3\delta \|x\|^2\|A\|_F$ with probability at least $1-8(m+1)\sum_i^{m} \rank (A_i)e^{-\mainconstant \delta^2 k^5 \gamma^4}$ where $\mainconstant$ is some universal constant. Following this, we have that
    \begin{align}
         \mathbf{d} & \leq \left| \ \mathbf{b}  - \left( \sum_{i=1}^{\numconstraints + 1} \alpha_i \begin{bmatrix}
            \mathbf{x_i}^{\top} A_1 \mathbf{x_i} + 3\delta \|\mathbf{x_i}\|^2\|A_1\|_F\\
           \vdots \\
           \mathbf{x_i}^{\top} A_m \mathbf{x_i} + 3\delta \|\mathbf{x_i}\|^2\|A_m\|_F
         \end{bmatrix} 
        \right)
         \right| \\
         & = \left| \ \mathbf{b}  - \left( \mathbf{b}
         + \sum_{i=1}^{\numconstraints + 1} \alpha_i \begin{bmatrix}
             3\delta \|\mathbf{x_i}\|^2\|A_1\|_F\\
           \vdots \\
        3\delta \|\mathbf{x_i}\|^2\|A_m\|_F
         \end{bmatrix}
         \right)
         \right| \\
          & = \begin{bmatrix}
             3\delta \|A_1\|_F(\sum_{i=1}^{\numconstraints + 1} \alpha_i\|\mathbf{x_i}\|^2)\\
           \vdots \\
        3\delta \|A_{\numconstraints}\|_F(\sum_{i=1}^{\numconstraints + 1} \alpha_i\|\mathbf{x_i}\|^2)
         \end{bmatrix}  
          \leq \begin{bmatrix}
             3(\numconstraints + 1)\delta w^2 \|A_1\|_F\\
           \vdots \\
        3(\numconstraints + 1)\delta w^2 \|A_m\|_F 
         \end{bmatrix} \leq \epsilon \mathbf{1}
    \end{align}
    with probability at least $1-8(m+1)\sum_i^{m} \rank (A_i)e^{-\mainconstant \delta^2 k^5 \gamma^4}$. We have done all calculations with the right-hand side of the quadratic form bound, but similar results are obtained from the other side. We take the maximum among all $A_i$ to define $\epsilon$ we obtain the desired result.  
\end{proof}
Note that this might limit the applicability of Proposition \ref{prop: feasibility of lifting} since implies that for some cases we cannot obtain a solution to the projected problem that we can lift back to the space of the original problem. 
In order to overcome this issue, we solve the following SDP (or its dual), which is guaranteed to have a feasible and bounded solution when projected.  \\
\begin{minipage}[t]{0.45\textwidth}
    \begin{align*}\tag{\(\mathcal{P}_{\lambda}\)} \label{primal lambda}
        \min \quad & \objectivesdp + \sum_{i=1}^{\numconstraints}  \bar{\lambda}_i \ubidual - \sum_{i=1}^{\numconstraints} \underline{\lambda}_i  \lbidual \\
        \text{s.t.} \quad & \constraintsdp  + \bar{\lambda}_i - \underline{\lambda}_i = b_i \quad  i= 1, \hdots, \numconstraints, \\
        & X \succeq 0, \\
        & \bar{\lambda}_i \geq 0, \  \underline{\lambda}_i \geq 0.
    \end{align*}
    \end{minipage}
    \hfill 
    \begin{minipage}[t]{0.45\textwidth}
    \begin{align*}\label{dual lambda}\tag{\(\mathcal{D}_{\lambda}\)}
    \max \quad & b^{\top} z \\
    \text{s.t.} \quad & \sum_{i=1}^{\numconstraints} A_i y_i \preceq C , \\
    & \lbidual \leq y_i \leq \ubidual \quad  i= 1, \hdots, \numconstraints.
    \end{align*}
\end{minipage} \\
We can now apply the positive map $ \phi_{\projector}: A \rightarrow \projectormap{A}$ as we did for problems \eqref{projected primal} and \eqref{projected dual}, and get the following primal/dual pair: \\
\begin{minipage}[t]{0.60\textwidth}    \begin{align*}\tag{\(\mathcal{RP}_{\lambda}\)} \label{projected primal lambda}
        \min \quad & \langle PC P^{\top}, Y \rangle + \sum_{i=1}^{\numconstraints}  \bar{\lambda}_i \ubidual - \sum_{i=1}^{\numconstraints} \underline{\lambda}_i  \lbidual \\
        \text{s.t.} \quad & \langle  \projectormap{A_i} ,Y \rangle  + \bar{\lambda}_i - \underline{\lambda}_i = b_i \quad  i= 1, \hdots, \numconstraints, \\
        & Y \succeq 0, \\
        & \bar{\lambda}_i \geq 0, \  \underline{\lambda}_i \geq 0.
    \end{align*} 
    \end{minipage}
    \hfill 
    \begin{minipage}[t]{0.40\textwidth}
    \begin{align*}\tag{\(\mathcal{RD}_{\lambda}\)} \label{projected dual lambda}
    \max \quad & b^{\top} z \\
    \text{s.t.} \quad & \sum_{i=1}^{\numconstraints} \projectormap{A_i} z_i \preceq PC P^{\top}, \\
    & \lbidual \leq z_i \leq \ubidual \quad  i= 1, \hdots, \numconstraints.
    \end{align*}
\end{minipage}
\begin{lemma}\label{lemma: bounded dual}
    Assume that problem \eqref{primal lambda} is bounded and feasible. 
    Then the projected problem \eqref{projected primal lambda} and its dual \eqref{projected dual lambda}  are also bounded and feasible. 
\end{lemma}
\begin{proof}
    Since the dual variables of the projected problem are all bounded, and the feasible region of the original problem is contained in the projected problem by Corollary \ref{cor: inclusion of feasible regions}, the projected dual (and projected primal) is bounded and feasible as desired. 
\end{proof}
Moreover, strong duality is also preserved after projection, as the following result shows. 
\begin{lemma}\label{lemm: strong duality of projected}
    Assume that Slater's condition holds for \eqref{primal lambda}, then it also holds for \eqref{projected primal lambda}.
\end{lemma}
\begin{proof}
    Since Slater's condition of \eqref{primal lambda} holds, we know that there is a matrix $\hat{X}$ and variables $\bar{\alpha}_i, \underline{\alpha}_i$ for $i= 1, \hdots, \numconstraints$ that satisfy the following system: 
    \begin{equation*}
    \begin{cases}
        \constraintsdp  + \bar{\lambda}_i - \underline{\lambda}_i = b_i \quad i= 1, \hdots, \numconstraints, \\
        X \succ 0, \\
        \bar{\lambda}_i > 0, \  \underline{\lambda}_i > 0.
    \end{cases}
    \end{equation*}
    Recall that positive definiteness is preserved by applying the map $\phi_{\projector}$. Let $\hat{Y} = P \hat{X} P^{\top}$ and $\bar{\beta}_i - \underline{\beta}_i = \langle A_i,\hat{X} \rangle  + \bar{\alpha}_i - \underline{\alpha}_i - \langle \projectormap{A_i},\hat{Y} \rangle$. 
    Since they satisfy the following equations
    \begin{equation*}
    \begin{cases}
        \langle \projectormap{A_i},\hat{Y} \rangle  + \bar{\beta}_i - \underline{\beta}_i = b_i, \quad i= 1, \hdots, \numconstraints, \\
        \hat{Y} \succ 0 \\
        \bar{\beta}_i > 0, \  \underline{\beta}_i > 0
    \end{cases}
    \end{equation*}
    there is a point in the interior of the feasible region of \eqref{projected primal lambda}, which means that the condition also holds. 
\end{proof}
Under the assumption that \eqref{primal} is bounded and feasible and that the dual variables $y_i$ are within the interval $(\ubidual, \lbidual)$, we can show that problem \eqref{primal} and \eqref{primal lambda} have the same optimal solution.
\begin{proposition}\label{prop: equivalence of primal of bounded dual} 
    If problem \eqref{primal} is bounded and feasible and all the dual variables $y_i$ are within the bounded intervals $(\ubidual, \lbidual)$, then the following problems have the same optimal solution \\
    \hfill \\
    \begin{minipage}[t]{0.45\textwidth}
    \begin{align*}\tag*{\eqref{primal}}
        \min \quad & \objectivesdp \\
        \text{s.t.} \quad & \constraintsdp  = b_i \quad  i= 1, \hdots, \numconstraints \\
        & X \succeq 0.
    \end{align*}
    \end{minipage}
    \hfill 
    \begin{minipage}[t]{0.45\textwidth}
    \begin{align*}\tag*{\eqref{primal lambda}}
        \min \quad & \objectivesdp + \sum_{i=1}^{\numconstraints}  \bar{\lambda}_i \ubidual - \sum_{i=1}^{\numconstraints} \underline{\lambda}_i  \lbidual \\
        \text{s.t.} \quad & \constraintsdp  + \bar{\lambda}_i - \underline{\lambda}_i = b_i \quad  i= 1, \hdots, \numconstraints \\
        & X \succeq 0 \\
        & \bar{\lambda}_i \geq 0, \  \underline{\lambda}_i \geq 0.
    \end{align*}
    \end{minipage}
\end{proposition}   
\begin{proof}
    We know by assumption that \eqref{dual} is bounded and feasible and that all dual variables $y_i$ are bounded by $(\ubidual, \lbidual)$. 
    Adding the constraints $ \lbidual \leq y_i \leq \ubidual$ to \eqref{dual} does not change the optimal value, which means that $\sigma \eqref{dual} = \sigma \eqref{dual lambda}$. 
    Now, to see that at optimality all elements of $\lambda = \{\bar{\lambda}_i, \underline{\lambda}_i\}_{i=1}^{\numconstraints}$ are equal to 0, consider the KKT conditions of \eqref{dual lambda}. 
    By the complementarity condition, either $ y_i = \ubidual$, or we must have the multiplier $\bar{\lambda}_i$ equal to $0$. 
    Since we know by construction that $ y_i \in (\lbidual, \ubidual)$, we have that $y_i \neq \ubidual$ and therefore $\bar{\lambda}_i=0$ for all $i$. 
    A similar reasoning applies to all $\underline{\lambda}_i$.  
\end{proof}

Note that Proposition \ref{prop: feasibility of lifting}, Proposition \ref{prop: feasibility of dual}, and Corollary \ref{cor: inclusion of feasible regions} also hold for \eqref{primal lambda} and \eqref{projected primal lambda}. 
However, the problem is that contrary to Proposition \ref{prop: equivalence of primal of bounded dual}, some elements of $\lambda = \{\bar{\lambda}_i, \underline{\lambda}_i\}_{i=1}^{\numconstraints}$ might not be equal to 0 at the optimal solution of \eqref{projected primal lambda}. 
This means that we cannot get a feasible solution of \eqref{primal} by solving \eqref{projected primal lambda} even though we have a feasible solution for \eqref{projected primal lambda}. 
However, this is not always the case, and for many practical applications the projected problem \ref{projected primal} is feasible for large enough projections.
However, we can establish some results and bounds of how to approximate the optimal solution of \eqref{primal lambda} (and consequently \eqref{primal} by Proposition \ref{prop: equivalence of primal of bounded dual}) by solving the projected problem \eqref{projected primal lambda}. 
To do this, we want to show that the optimal objective value of \eqref{primal} is close to the optimal objective of \eqref{projected primal lambda} up to some approximation error. 
Namely, we want to find $E$ such that 
\begin{equation}\label{bounding original objective}
    \sigma\eqref{projected primal lambda} - E \leq \sigma\eqref{primal} \leq \sigma\eqref{projected primal lambda},
\end{equation}
where $\sigma$ denotes the optimal objective value of the problems.

Recall that $\sigma\eqref{primal} = \sigma\eqref{primal lambda}$ under the assumption that all the dual variables of \eqref{primal} are bounded. 
This means that we can find $E$ for $\sigma\eqref{projected primal lambda} - E \leq \sigma\eqref{primal lambda} \leq \sigma\eqref{projected primal lambda}$ instead of \eqref{bounding original objective}. 
The right-hand side is straightforward, since for any projected solution we can find a feasible solution of the original problem with the same objective value. 
To find the left-hand side $\sigma\eqref{projected primal lambda} - E \leq \sigma\eqref{primal lambda}$, we follow a similar approach to that presented in \cite{DAmbrosio2020}, but adapted for the SDP case.
\begin{proposition}\label{prop: optimal original is better than projected}
    Let $\sigma \eqref{primal lambda}$ be the objective value of the original problem, and $\sigma \eqref{projected primal lambda}$ be the objective value of the projected problem. Then we have that 
    \begin{equation*}
        \sigma \eqref{primal lambda} \leq \sigma\eqref{projected primal lambda}.
    \end{equation*}
\end{proposition}
\begin{proof}
    Let $(Y^*, \lambda^*)$ be the optimal solution of the projected problem. Then, the objective value will be $\sigma\eqref{projected primal} = \langle PCP^{\top}, Y^* \rangle + \sum_{i=1}^{\numconstraints}  \bar{\lambda}_i^* \ubidual - \sum_{i=1}^{\numconstraints} \underline{\lambda}_i^*  \lbidual = \sum_{(i,j)} (PCP^{\top})_{i,j}Y^*_{i,j} + \sum_{i=1}^{\numconstraints}  \bar{\lambda}_i^* \ubidual - \sum_{i=1}^{\numconstraints} \underline{\lambda}_i^*  \lbidual $. Let $\hat{X} = P^{\top}Y^*P$. Then the objective evaluated at feasible $\hat{X}$ is $\langle C, P^{\top}Y^*P \rangle + \sum_{i=1}^{\numconstraints}  \bar{\lambda}_i^* \ubidual - \sum_{i=1}^{\numconstraints} \underline{\lambda}_i^*  \lbidual = C_{i,j} (P^{\top}V^*P)_{i,j}+ \sum_{i=1}^{\numconstraints}  \bar{\lambda}_i^* \ubidual - \sum_{i=1}^{\numconstraints} \underline{\lambda}_i^*  \lbidual = \sigma\eqref{projected primal lambda}$. 
    Since the optimal value of the projected problem can always be mapped to a feasible solution of the original problem with the same objective value, we have that $\sigma \eqref{primal lambda} \leq \sigma \eqref{projected primal lambda}$, as desired. 
\end{proof}
\begin{theorem}[Approximate Optimality]\label{thm: main projected sdp optimality theorem}
    Let $P \in \reals{\dimprojector \times \dimoriginal}$ be a RP as that of Lemma \ref{lemma: frobenious and nuclear bound}, $(X^*, \lambda^*)$ be the optimal solution of \eqref{primal lambda}, and $z^*$ be the optimal solution of \eqref{projected dual lambda}. Then, with probability at least $1-8(r_C r_{X^*} + \sum_i^{\numconstraints} r_{A_i} r_{X^*}) e^{-\mainconstant \epsilon^2 k^5 \gamma^4}$ we have 
    \begin{equation*}
        \sigma\eqref{projected primal lambda} - E \leq \sigma\eqref{primal lambda} \leq \sigma\eqref{projected primal lambda}, 
    \end{equation*}
where $E = 3 \epsilon \|C\|_F\|X^*\|_* + \sum_i^{\numconstraints} \max \{ 3 z_i^*  \epsilon \|A_i\|_F\|X^*\|_*,-  3 z_i^*  \|A_i\|_F\|X^*\|_*\}$, $r_C$ is the rank of $C$, $r_{X^*}$ is the rank of $X^*$, and $r_{A_i}$ denotes the ranks of $A_i$.
\end{theorem}

\begin{proof}
    Let $(X^*, \lambda^*)$ be the optimal solution to \ref{primal lambda} and $Y' = \projectormap{X^*}$ be its projection. Then, we start by constructing the following auxiliary problem.  
    \begin{align*}\tag{$\mathcal{RP}_{\lambda}'$}\label{aux projected primal}
    \min \quad & \langle PCP^{\top}, Y \rangle + \sum_i^{\numconstraints}  \bar{\lambda}_i \ubidual - \sum_i^{\numconstraints} \underline{\lambda_i}  \lbidual \\
    \text{s.t.} \quad & \langle \projectormap{A_i},Y \rangle  + \bar{\lambda}_i - \underline{\lambda_i} = b_i - ((\langle A_i, X^* + \bar{\lambda}_i^* - \underline{\lambda_i}^*) \rangle - (\langle \projectormap{A_i}, Y' \rangle + \bar{\lambda}_i^* - \underline{\lambda_i}^*), \\
    & Y \succeq 0.
\end{align*}

The projected solution $(Y', \lambda^*)$ is feasible for the auxiliary problem. 
To see this, note that since $(X^*, \lambda^*)$ is feasible for \eqref{primal lambda}, $\langle A_i,X^* \rangle  + \bar{\lambda}_i^* - \underline{\lambda_i}^* = b_i$ for all $i$. 
Therefore, $\langle \projectormap{A_i},Y \rangle + \bar{\lambda}_i - \underline{\lambda_i} = b_i - b_i + \langle \projectormap{A_i}, Y' \rangle + \bar{\lambda}_i^* - \underline{\lambda_i}^* $, which is satisfied by ($Y', \lambda^*)$. 
Since $(Y', \lambda^*)$ is feasible for the auxiliary problem and we are minimising, we know that the objective value evaluated at $Y'$ is greater than or equal to the optimal objective of the auxiliary problem, namely that
\begin{equation*}
    \langle PCP^{\top},Y' \rangle + \sum_i^{\numconstraints}  \bar{\lambda}_i^* \ubidual - \sum_i^{\numconstraints} \underline{\lambda_i}^*  \lbidual  \geq \sigma\eqref{aux projected primal}.
\end{equation*}
Now, we want to find $M_1, M_2$ such that 
\begin{equation*}
    M_1 \geq \langle PCP^{\top},Y' \rangle + \sum_i^{\numconstraints}  \bar{\lambda}_i^* \ubidual - \sum_i^{\numconstraints} \underline{\lambda_i}^*  \lbidual \geq \sigma\eqref{aux projected primal} \geq M_2.
\end{equation*}
To find $M_1$, note that from Lemma \ref{lemma: frobenious and nuclear bound}, $\langle PCP^{\top},Y' \rangle \leq \langle C,X^* \rangle + 3 \epsilon \|C\|_F\|X^*\|_*$ with probability at least $1-8(r_C r_{X^*}) e^{-\mainconstant \epsilon^2 k^5 \gamma^4}$, where $\langle C,X^* \rangle + \sum_i^{\numconstraints}  \bar{\lambda}_i^* \ubidual - \sum_i^{\numconstraints} \underline{\lambda_i}^*  \lbidual = \sigma\eqref{primal lambda}$ and $\lambda^* \geq 0$. 
Therefore, we have that 
\begin{equation*}
    M_1 = \sigma\eqref{primal lambda} + 3 \epsilon \|C\|_F\|X^*\|_*.
\end{equation*}
Now, to find $M_2$ consider the dual of \eqref{aux projected primal}:
\begin{align*}\tag{$\mathcal{RD}_{\lambda}'$}\label{aux projected dual}
    \max \quad & \sum_i^{\numconstraints} y_i (b_i - ((\langle A_i, X^* \rangle + \bar{\lambda}_i^* - \underline{\lambda_i}^*) - (\langle \projectormap{A_i}, Y' \rangle + \bar{\lambda}_i^* - \underline{\lambda_i}^*)) \\
    \text{s.t.} \quad & \sum_{} \projectormap{A_i} y_i \preceq PC P^{\top} \\
    & \lbidual \leq y_i \leq \ubidual. \\
\end{align*}
By weak duality of the auxiliary problem, we have that $\sigma\eqref{aux projected primal} \geq \sigma\eqref{aux projected dual}$. 
Now, the original projected dual \eqref{projected dual lambda} and the auxiliary projected dual \eqref{aux projected dual} have the same feasible region, which means that the optimal solution of the projected dual $z^*$ is feasible for the auxiliary dual. 
This implies that $\sigma\eqref{aux projected dual} \geq \sum_i^{\numconstraints} z_i^* (b_i - ((\langle A_i, X^* \rangle + \bar{\lambda}_i^* - \underline{\lambda_i}^*) - (\langle \projectormap{A_i}, Y' \rangle + \bar{\lambda}_i^* - \underline{\lambda_i}^*)))$. Therefore, we have that 
\begin{align*}
    \sigma\eqref{aux projected primal} \geq  & \sum_i^{\numconstraints} z^*_i (b_i - ((\langle A_i, X^* \rangle + \bar{\lambda}_i^* - \underline{\lambda_i}^*) - (\langle \projectormap{A_i}, Y' \rangle + \bar{\lambda}_i^* - \underline{\lambda_i}^*))) \\
    = & \ \sigma\eqref{projected dual lambda} - \sum_i^{\numconstraints} z^*_i (\langle A_i, X^* \rangle - \langle \projectormap{A_i}, Y' \rangle ) \\
    = & \ \sigma\eqref{projected primal lambda} - \sum_i^{\numconstraints} z^*_i (\langle A_i, X^* \rangle - \langle \projectormap{A_i}, Y' \rangle )
\end{align*}
Now, by Lemma \ref{lemma: frobenious and nuclear bound}, $|\langle A_i, X^* \rangle - \langle \projectormap{A_i}, Y' \rangle |$ is bounded by $3 \epsilon \|A_i\|_F\|X^*\|_* $ for each $i$ with probability at least $1-8 (r_{A_i} r_{X^*}) e^{-\mainconstant \epsilon^2 k^5 \gamma^4}$. 
Following this, we have that 
\begin{equation*}
    M_2 = \sigma\eqref{projected primal lambda}) - E_1
\end{equation*}
where $E_1 = \sum_i\max \{ 3 z_i^*  \epsilon \|A_i\|_F\|X^*\|_*, - 3 z_i^*  \epsilon \|A_i\|_F\|X^*\|_*\}$.
Note that $E_1$ is never non-negative. If the dual variable $z_i^*$ is negative, then the right-hand side term of $E_1$ will be chosen. 
On the other hand, if $z_i^*$ is positive, then the left term will be chosen. 
Adding everything together we get that 
\begin{equation*}
    \sigma\eqref{primal lambda} + 3 \epsilon \|C\|_F\|X^*\|_*
    \geq \sigma\eqref{projected primal lambda} - E_1
\end{equation*}
with probability at least $1-8(r_C r_{X^*} + \sum_i^{\numconstraints} r_{A_i} r_{X^*}) e^{-\mainconstant \epsilon^2 k^5 \gamma^4}$.
Therefore, we obtain the desired result that 
\begin{equation*}
    \sigma\eqref{primal lambda} \geq \sigma\eqref{projected primal lambda} - E
\end{equation*}
for $E =  3 \epsilon \|C\|_F\|X^*\|_* + E_1$ with probability at least $1-8(r_C r_{X^*} + \sum_i^{\numconstraints} r_{A_i} r_{X^*}) e^{-\mainconstant \epsilon^2 k^5 \gamma^4}$.
\end{proof}

Recall that since \eqref{primal lambda} and \eqref{primal} share the same objective, this bound also applies to the optimal value of \eqref{primal}. 
This is captured in the following corollary. 
\begin{corollary} \label{cor: optimality bounds of original problem}
    Let $P \in \reals{\dimprojector \times \dimoriginal}$ be a RP as that of Lemma \ref{lemma: frobenious and nuclear bound}, $(X^*, \lambda^*)$ be a solution to \eqref{primal lambda}, and $z^*$ be a solution of \eqref{projected dual lambda}. Then, with probability at least $1-8(r_C r_{X^*} + \sum_i^{\numconstraints} r_{A_i} r_{X^*}) e^{-\mainconstant \epsilon^2 k^5 \gamma^4}$ we have 
    \begin{equation*}
        \sigma\eqref{projected primal lambda} - E \leq \sigma\eqref{primal} \leq \sigma\eqref{projected primal lambda}, 
    \end{equation*}
where $E = 3 \epsilon \|C\|_F\|X^*\|_* + \sum_i^{\numconstraints} \max \{ 3 z_i^*  \epsilon \|A_i\|_F\|X^*\|_*,-  3 z_i^*  \|A_i\|_F\|X^*\|_*\}$, and $r_C$ is the rank of $C$, $r_{X^*}$ is the rank of $X^*$, and $r_{A_i}$ are the rank of $A_i$.
\end{corollary}

It is important to recall again that by solving \eqref{projected primal lambda} we do not necessarily get a feasible solution for \eqref{primal}. 
Moreover, we need to assume that the dual variables of the original problem are bounded and we need to know these bounds in advance, which is not always straightforward. 
However, there are some applications such as polynomial optimization that satisfy this assumption and where having a feasible solution to the original problem is not necessary. 
For most practical applications computational experiments show that if the projected dimension is large enough and the projector is dense enough the projection is almost always feasible.
However, knowing this in advance for general problems is hard. 
There are some applications of SDP relaxations of QCQP such as \textsc{maxcut} and \textsc{max-2-sat} where the projected problem is almost always feasible. 
This allows us to not need dual bounds in advance, and also to retrieve a feasible solution for the original problem after projection. 

Lastly, we want to present an updated lower bound on $\dimprojector$ that ensures that the probability of Corollary \ref{cor: optimality bounds of original problem} stays close enough to 1.
Assume we want to probability to be at least $1-\delta$, then we have that 
\begin{equation*}
    \dimprojector \geq \left(\frac{\ln(8(r_C r_{X^*} + \sum_i^{\numconstraints} r_{A_i} r_{X^*})/\delta))}{\mainconstant \epsilon^2 \gamma^2}\right)^{\frac{1}{5}}.
\end{equation*}

\section{Computational Complexity}\label{sec: computational complexity}
In this section, we aim to answer the question of how much it is saved in terms of computational complexity when solving the projected problem. 
We compare the complexity of solving the two problems using IPMs, but we also take into account the complexity of projecting down the original problem.
It is important to point out that we are comparing the time complexity of the original problem and an approximation.
However, knowing how much faster the projected problem is can still be very valuable, even if the solution quality is worse. 
Big saving of the projected problem, even at the cost of accuracy, can motivate the use of random projections embedded into iterative algorithms. 

Let us start by considering the differences in complexity when solving a semidefinite programming problem and its projection using IPMs. 
For this solution approach and a problem of matrix size $\dimoriginal$ and $m$ constraints, the complexity of each iteration is $\bigO(m \dimoriginal^3 + m^2 \dimoriginal^2 + m^3)$, and we need $\bigO(\sqrt{\dimoriginal} \log(1/\epsilon{}))$ iterations to get an $\epsilon$ solution. 
Note that we do not assume that $\dimoriginal > m$, since for some applications such as polynomial optimization we might have more constraints than the size of the matrix variable. 
This means that it takes 
\begin{equation}\label{original complexity IPM}
    \bigO((m\dimoriginal^{3.5} + m^2\dimoriginal^{2.5} + \sqrt{\dimoriginal}m^3) \log(1/\epsilon{}))
\end{equation}
to solve the problem to $\epsilon$ accuracy. 
Recall that the projected problem has a variable matrix $Y \in \reals{\dimprojector \times \dimprojector}$ and $m$ constraints. 
However, since  $\dimprojector = \bigO(\ln \dimoriginal)$ the complexity of solving the projected problem to $\epsilon$ accuracy is 
\begin{equation}\label{projected complexity IPM}
    \bigO((m\ln^{3.5}\dimoriginal + m^2\ln^{2.5}\dimoriginal + \sqrt{\ln \dimoriginal}m^3) \log(1/\epsilon{})).
\end{equation}
Now, it is necessary to clarify that the accuracy $\epsilon$ in \eqref{original complexity IPM} and \eqref{projected complexity IPM} does not refer to the same solution. Since the projected problem is an approximation of the original problem, \eqref{projected complexity IPM} refers to the complexity of getting $\epsilon$ accuracy with respect to the approximation,  whose value is worse than that of the original SDP. 

Moreover, we also need to take into account the  complexity of projecting all the matrices in the problem.
The methodology proposed in this paper can be broken down into three main steps: sampling the random projector, projecting all the matrices of the problem with this projector, and solving the projected problem. 
The computational complexity of sampling the projector $\projector$ is roughly $\bigO(1)$ per entry, which makes the total complexity $\bigO(\dimprojector \dimoriginal)$ for the projected dimension $\dimprojector$ and the original dimension $\dimoriginal$. We do not need to perform a matrix transpose operator to get $\projectorT$ since we can just swap the indices of entries of $P$.
For a problem with $m$ constraints we need to project $m+1$ matrices of size $\dimoriginal \times \dimoriginal$. 
Matrices are projected using the positive map $ \phi_{\projector}:  \reals{\dimoriginal \times \dimoriginal} \rightarrow  \reals{\dimprojector \times \dimprojector}$ that maps $ A \in \reals{\dimoriginal \times \dimoriginal}$ to $ \projectormap{A} \in \reals{\dimprojector \times \dimprojector}$.
The complexity of applying this map is $\bigO(\dimprojector \dimoriginal^2 + \dimoriginal \dimprojector^2)$.
Since we have to project $(m+1)$ matrices of size $\dimoriginal \times \dimoriginal$, the total complexity is $\bigO((m+1)(kn^2+nk^2))$. 
If matrix $P$ has $r$ non-zero entries per column and $P^{\top}$ has $z$, then the complexity of $PA$ is $\bigO(knr)$ and the complexity of $\projectormap{A}$ is $\bigO(knr + knz)$.
Again, since we have to project $m + 1$ of these matrices, the total complexity of projecting using a sparse projector is $\bigO((m+1)(knr + knz))$. 
Note that for some applications such as \textsc{maxcut} and \textsc{Max-2-sat}, the number of non-zero entries per column is 1, so the complexity is even smaller. 
For this case, projecting a matrix $A_i$ with 0s everywhere except a 1 in entry $A_{ii}$ is equivalent to taking the outer product of the $i$th columns of $\projector$. 
The complexity of the outer product of two vectors is $\bigO(n^2)$ where $n$ is the length of the vector. When the vector is sparse and only $r$ entries are non-zero, the complexity is $\bigO(r^2)$.

Adding everything together we have that the original complexity to achieve $\epsilon$ accuracy (for the original problem) is 
\begin{equation}
    \bigO((m\dimoriginal^{3.5} + m^2\dimoriginal^{2.5} + m^3\sqrt{\dimoriginal}) \log(1/\epsilon{}))
\end{equation}
while the complexity to solve the projection up to $\epsilon$ (with respect to the projected objective, which is known to be close to the original one with high probability) is 
\begin{align}
    \bigO(\dimprojector \dimoriginal) + \bigO((m+1)(kn^2+nk^2)) + \bigO((m\ln^{3.5}\dimoriginal + m^2\ln^{2.5}\dimoriginal + m^3\sqrt{\ln \dimoriginal}) \log(1/\epsilon{})) \\
    =
    \bigO(m \dimoriginal^2 \ln \dimoriginal ) + \bigO((m\ln^{3.5}\dimoriginal + m^2\ln^{2.5}\dimoriginal +  m^3\ln^{0.5}\dimoriginal) \log(1/\epsilon{}))
\end{align}
Also, recall that to retrieve a feasible solution to the original problem we need to apply the inverse of the congruence map. 
However, this does not change the complexity since it would be $\bigO((m+1)(kn^2+nk^2))$, which is already considered. 

By using \textit{soft-O} notation, which hides polylogarithmic factors, we are able to compare them easier. 
For the original problem, the complexity remains the same, i.e. $\Tilde{\bigO}((m\dimoriginal^{3.5} + m^2\dimoriginal^{2.5} + m^3\dimoriginal^{0.5}) \log(1/\epsilon{}))$.
However, for the projected case, the complexity reduces to 
\begin{equation}
    \Tilde{\bigO}(m \dimoriginal^2) + \Tilde{\bigO}(m^3\log(1/\epsilon{})).
\end{equation}
Note that for the projection method, we have the constant term $\Tilde{\bigO}(m \dimoriginal^2)$, which comes from the matrix multiplication before projection. 
In reality, since the projector is sparse, the complexity of matrix multiplication can be reduced. 
Moreover, computational experiments show that we can usually increase the sparsity parameter sub-linearly with respect to the dimension.
This means that the larger the dimension, the sparser we can make the projector, and this relation does not seem to show a linear behavior. 
Our belief is that it might be logarithmic. 
However, proving this formally is out of the scope of this paper. 
Lastly, it is important to bear in mind that one of the main drawbacks of interior point methods is running into memory issues. These memory problems can be drastically reduced by the use of random projections.

\section{Computational Experiments}\label{sec: computational experiments}

We now present some computational experiments and test our method on different applications. 
Recall that the optimality bound is additive in terms of the constraints, and depends on the nuclear and Frobenious norm of the data matrices and the solution matrix.
This means that projections of SDP problems with a small number of constraints and data matrices with a small nuclear norm or small Frobenious norm will work better. 
Following this, we apply our method to the semidefinite relaxations of \textsc{maxcut} and \textsc{max-2-sat}, since the constraint data matrices of these problems have nuclear and Frobenious norm equal to one. 
The norm of the objective matrix $C$ is problem-dependent, but one can estimate its dimension depending on the parameters of the problems such as the number of variables and the ratio of clauses for \textsc{max-2-sat}, or the number of vertices and the density of the graph for \textsc{maxcut}. 
For these problems, we show that a projected SDP relaxation that considerably reduces the dimension of the SDP problem can be solved while getting reasonable bounds for the original problems. 
We also explore applications of our method to binary polynomial optimization problems. 
In particular, we show that it can be applied to solve the second level of the Lasserre Hierarchy for some instances of the Stable Set Problem.

We present \textsc{maxcut} results for two types of graphs. 
First, we explore the G-set dataset taken from \cite{gset}. However, this first experiment shows that our method performs much better on unweighted graphs and that the larger the graph, the better the quality of the projection. 
Following this, we also explore some larger unweighted random graphs generated using the rudy\footnote{\href{https://www-user.tu-chemnitz.de/∼helmberg/rudy.tar.gz}{https://www-user.tu-chemnitz.de/$\sim$helmberg/rudy.tar.gz}} graph generator by G. Rinaldi. 
For \textsc{max-2-sat}, we use randomly generated formulas generated from a DIMACS\footnote{\href{https://dimacs.rutgers.edu/pub/challenge/satisfiability/cotributed/selman/}{https://dimacs.rutgers.edu/pub/challenge/satisfiability/cotributed/selman/}} challenge. 
For the stable set problem, we study the complement of generalised Petersen graphs, the complement of Helm graphs, the complement of Jahangir's graphs, and the complement of the graphs presented in \cite{Dobre2015}. 
All computational experiments were conducted on a Dell PowerEdge R740 server, equipped with 4 Intel Gold 6234 processors of 3.3GHz, 8 cores, and 16 threads. 
The main implementation is in Python \cite{python}, and we use Mosek as the SDP solver with default configuration \cite{mosek}.  

\subsection{Semidefinite Relaxations}
Semidefinite programming programs are particularly useful when approximating hard classes of optimization problems such as non-convex quadratic programs \cite{Fujie1997, Boyd1997, Nesterov2000}. 
These relaxations are more powerful than linear programming (LP) relaxations and often provide a good approximation of the original problem. 
In this subsection, we present computational results for two of these problems: \textsc{maxcut} and \textsc{max-2-sat}.
The maximum cut problem (\textsc{maxcut}) is the problem of finding the partition of a graph where the number of edges connecting the two new sets is maximised. This problem is known to be NP-hard \cite{Lewis1983}. Let $G = (E, V)$ be a non-directed graph where $E$ is the set of edges and $V$ is the set of vertices. We define the variables $x_v$ for $v \in V$, and $z_{uv}$ for $uv \in E$. We want to maximize the number of edges connecting vertices belonging to different sets after we cut the set by half. The SDP relaxation of \textsc{maxcut} is 
\begin{align*}
\max \ & \langle L, X \rangle\\
\text{s.t.} \ & \langle A_i, X \rangle = 1 \quad i = 1, \dots, n \\
& X \succeq 0
\end{align*} 
where $L \in \reals{\dimoriginal \times \dimoriginal}$ is the Laplacian matrix of the graph, and $A_i$ are the matrices that makes sure that the diagonal element ith of $X$ is equal to one. 
On the other hand, the maximum 2-satisfiability problem (\textsc{max-2-sat}) seeks to find an assignment to the variables of a logic formula such that the maximum number of clauses is satisfied, where we have at most two variables per clause. An example of such a formula with three clauses and 5 variables is:
\begin{equation}
    (x_1 \lor \neg x_3 ) \land (x_2 \lor x_4) \land (\neg x_1 \lor x_5)
\end{equation}
The formulation of the SDP relaxation of \textsc{max-2-sat} is very similar to that of \textsc{maxcut}. The only difference is that instead of $L$ we have a matrix $W$ that represents the sum of the clauses in the formula. 
Both \textsc{maxcut} and \textsc{max-2-sat} share the same feasible region, i.e. psd matrices with 1s in the diagonal. 
Applying the projection $\projectormap{A_i}$ to these data matrices is equivalent to taking the outer product of the $i$th column of the projector $P$. 
This makes computing the projections much more efficient, and also allows us to investigate the feasibility of the projected problem in more detail. 
Moreover, the larger the problem size, the sparser we can make the projector without getting an infeasible projection. 
Following this, we use different projectors depending on the problem size: the classic sparse Achlioptas projector \cite{Achlioptas2001}, which has a density of 0.3, is used for problems of size 800 and the sparse projector from \cite{DAmbrosio2020} with densities 0.05 and 0.03 is used for matrices of size 5000 and 7000, respectively. 

We start testing \textsc{maxcut} on the standard G-set library, a collection of 67 medium to large-scale graphs \cite{gset}. 
Results for G-set graphs with 800 nodes are shown in Table \ref{tab: G-800}. 
Graphs G1 to G5 are random graphs with 800 vertices, unit weights and a density of 6\%. Graphs G6 to G10 are the same graphs with edge weights taken randomly from \{-1, 1\}.
Graphs G11 to G13 are toroidal grids that also have random edge weights from \{-1, 1\} and 800 vertices, while G14 to G17 are composed of the union of two (almost maximal) planar graphs with unit weights.
Graphs G18 to G21 are 800 vertices planar graphs but with random \{-1, 1\} edge weights.
The rest of the graphs in the dataset follow the same generation pattern
 but with more nodes. 
In this paper, we only present the results for graphs of size 800 since the insights we get from larger graphs are similar.

\begin{table}
    \captionof{table}{Relative time (\%) and relative quality (\%) with respect to the original relaxation of \textsc{maxcut} of G-graphs with 800 nodes using an Achlioptas projector \cite{Achlioptas2001}} 
    \begin{tabular}{lrrrrrrrrrrrrrrr} 
        \toprule
        & && \multicolumn{2}{c}{10\% projection} && \multicolumn{2}{c}{20\% projection} \\
        \cmidrule{4-5} \cmidrule{7-8}
        \rule{0pt}{10pt} 
        Instance  & $m$ && Time & Quality && Time & Quality \\
        \midrule
        G1 & 19176 && 52.78 &  82.79  && 261.86 &  86.99  \\
        G2 & 19176 && 46.28 &  82.61  && 222.94 &  86.63  \\
        G3 & 19176 && 42.57 &  82.58  && 259.63 &  86.65  \\
        G4 & 19176 && 49.62 &  82.56  && 273.76 &  86.44  \\
        G5 & 19176 && 47.65 &  82.51  && 200.39 &  86.67  \\
        \hline
        G6 & 19176 && 52.67 &  18.52  && 255.35 &  37.09  \\
        G7 & 19176 && 55.69 &  14.33  && 238.72 &  34.08  \\
        G8 & 19176 && 56.10 &  13.16  && 384.84 &     33.17  \\
        G9 & 19176 && 62.13 &  14.96  && 316.96 &     34.67  \\
        G10 & 19176 && 45.42 &  14.44  && 292.81 &     34.33  \\
        \hline
        G11 & 1600 && 51.93 &  22.17  && 243.63 &     44.71  \\
        G12 & 1600 && 56.37 &   19.97   && 313.99 &     42.88  \\
        G13 & 1600 && 47.95 &  20.74   && 319.77 &     42.62  \\
        \hline
        G14 & 4694 && 45.97 &   79.77  && 262.67 &     85.80  \\
        G15 & 4661 && 43.82 &   79.53   && 191.34 &     85.61  \\
        G16 & 4672 && 47.32 &     79.98  && 233.20 &     85.94  \\
        G17 & 4667 && 39.36 &     79.74  && 197.72 &     85.89  \\
        \hline
        G18 & 4694 && 54.25 &   20.80  && 242.70 &     40.18  \\
        G19 & 4661 && 46.54 &     14.44   && 231.37 &     35.33  \\
        G20 & 4672 && 51.99 &     15.59  && 236.06 &     37.68  \\
        G21 & 4667 && 53.43 &     14.63  && 216.14 &     36.51  \\
        \bottomrule
    \end{tabular}
    \label{tab: G-800}
\end{table}

In all the computational experiments we present results in terms of relative time and relative quality percentages. 
Relative time of the projected problem is the proportion of time taken to solve the projected problem compared to solving the original SDP problem. 
For example, a relative time of 5\% implies that we can solve the projected problem in 5\% of the time taken to solve the original sdp.
Relative quality measures the quality of the bound obtained with the projection with respect tot he quality of the original sdp. 
From these experiments, we can see that our method performs much better on unweighted graphs. For graphs G1 to G5, which are random unweighted graphs with a density of 6\%, we can solve a problem whose variable matrix is a tenth of the size of the original relaxation while getting an approximation quality of over 80\%. For the other types of unweighted graphs, from G14 to G17, similar results are obtained. However, results for weighted graphs are not as positive. This might be because the norm of the Laplacian matrix of weighted graphs is larger. Since the computational times with the Achlioptas projector were generally slower than those of sparse matrices, we switch to using the projectors proposed by D'Ambrosio et al. in \cite{DAmbrosio2020} for the rest of the experiments. The reason we have not presented those results directly in Table \ref{tab: G-800} is so that they can be compared with future performance. 

Following this, we now present in Table \ref{tab: main rudy-5000} results for unweighted graphs of dimension 5000 generated using rudy. 
We show results for different graphs generated with two different seeds for densities of 5\% or 10\%, 15\%, 20\%, 30\%, 40\% and 50\%. 
We present analogous results for dimension 7000 in Table \ref{tab: main rudy-7000}. 
For these, we just present results for 10\% since this projection has more computational savings. 
We can get an approximation quality of over 95\% by solving the problem in a third of the time for 5000 node graphs and half of the time for 5000 node graphs.
The longer time for larger graphs might be due to the tuning of the sparsity parameter. As we will see later in Table \ref{tab: sparsity}, when the projector is tuned correctly we can solve the problem with high accuracy a fourth of the time even for large graphs. 

In tables \ref{tab: main rudy-5000} and \ref{tab: main rudy-7000} we can see how random projections seem to work better for denser graphs. 
Overall, for the same density, the projection quality is better for larger graphs. This is illustrated better in Figure \ref{fig: quality}, which shows the quality of the projections for several graphs with density 30\% and increasing size. 
These results are taken from an average of 5 different graphs for each size. 

From current results, we can assume that the quality of projections will keep increasing, or at least be maintained, for even larger graphs. 
Very large graphs cannot be directly solved using Mosek since it runs into memory problems so they are harder to compare. 
We can instead use random projections to solve large-scale \textsc{maxcut} problems generated by rudy with a good approximation quality. 
However, the sparsity parameter of the projector is key to getting a feasible approximation that is solved in a reasonable amount of time. 
Too dense projectors might lead to too long computational times. 
An analysis of this sparsity parameter for a 10k nodes graph is presented in Table \ref{tab: sparsity}. 

\begin{table}[!htbp]
        \captionof{table}{Relative time (\%) and relative quality (\%) with respect to the original relaxation of \textsc{maxcut} of 5000 vertices rudy graphs of different densities (\%) for 10\% projection using a 5\% sparse projector from D'Ambrosio et al. \cite{DAmbrosio2020}}
        \begin{tabular}{lrrrrrrrrrrrrrrr} 
            \toprule
            & & && \multicolumn{2}{c}{10\% projection} \\
            \cmidrule{5-6}
            \rule{0pt}{10pt} 
            Instance & $m$ & Density && Time & Quality \\
            \midrule
             mcp-5000-5-1 &     624875 & 5                  & & 37.53 & 91.16 \\
             mcp-5000-5-2 &     624875 & 5                  & & 33.07 & 91.14 \\
             mcp-5000-10-1 &    1249750 & 10                 & & 30.88 & 93.69 \\
             mcp-5000-10-2 &    1249750 & 10                 & & 24.32 & 93.68 \\
             mcp-5000-15-1 &    1874625 & 15                 & & 29.83 & 94.89 \\
             mcp-5000-15-2 &    1874625 & 15                 & & 26.15 & 94.90 \\
             mcp-5000-20-1 &    2499500 & 20                 & & 34.59 & 95.67 \\
             mcp-5000-20-2 &    2499500 & 20                 & & 27.68 & 95.68 \\
             mcp-5000-30-1 &    3749250 & 30                 & & 27.41 & 96.65 \\
             mcp-5000-30-2 &    3749250 & 30                 & & 28.91 & 96.66 \\
             mcp-5000-40-1 &    4999000 & 40                 & & 32.97 & 97.29 \\
             mcp-5000-40-2 &    4999000 & 40                 & & 29.05 & 97.30 \\
             mcp-5000-50-1 &    6248750 & 50                 & & 23.72 & 97.78 \\
             mcp-5000-50-2 &    6248750 & 50                 & & 29.66 & 97.78 \\
            \bottomrule
        \end{tabular}
        \label{tab: main rudy-5000}
    \end{table}

\begin{table}[!htbp]
    \captionof{table}{Relative time (\%) and relative quality (\%) with respect to the original relaxation of \textsc{maxcut} of 7000 vertices rudy graphs of different densities (\%) for 10\% projection using a 4\% sparse projector from D'Ambrosio et al. \cite{DAmbrosio2020}} 
    \begin{tabular}{lrrrrrrrrrrrrrrr} 
        \toprule
        & & && \multicolumn{2}{c}{10\% projection} \\
        \cmidrule{5-6}
        \rule{0pt}{10pt} 
        Instance & $m$ & Density && Time & Quality \\
        \midrule
             mcp-7000-5-1 &    1224825 & 5                  & & 54.24 & 92.62 \\
             mcp-7000-5-2 &    1224825 & 5                  & & 61.82 & 92.61 \\
             mcp-7000-10-1 &    2449650 & 10                 & & 64.24 & 94.77 \\
             mcp-7000-10-2 &    2449650 & 10                 & & 56.90 & 94.76 \\
             mcp-7000-15-1 &    3674475 & 15                 & & 55.91 & 95.78 \\
             mcp-7000-15-2 &    3674475 & 15                 & & 68.48 & 95.79 \\
             mcp-7000-20-1 &    4899300 & 20                 & & 55.70 & 96.43 \\
             mcp-7000-20-2 &    4899300 & 20                 & & 54.67 & 96.43 \\
             mcp-7000-30-1 &    7348950 & 30                 & & 56.24 & 97.24 \\
             mcp-7000-30-2 &    7348950 & 30                 & & 54.33 & 97.24 \\
             mcp-7000-40-1 &    9798600 & 40                 & & 54.45 & 97.78 \\
             mcp-7000-40-2 &    9798600 & 40                 & & 48.72 & 97.77 \\
             mcp-7000-50-1 &   12248250 & 50                 & & 35.32 & 98.18 \\
             mcp-7000-50-2 &   12248250 & 50                 & & 53.69 & 98.17 \\
        \bottomrule
    \end{tabular}
    \label{tab: main rudy-7000}
\end{table}

\begin{table}[!htbp]
    \captionof{table}{Comparison of relative time (\%) and relative quality (\%) with respect to the original SDP relaxation for different sparsity parameters (\%) of the sparse projectors from D'Ambrosio et al. \cite{DAmbrosio2020} for a 5\% and 10\% projection of a 10.000 nodes  rudy graph with 20\% edge density, where infeasibility is denoted by -}
    \begin{tabular}{rrrrrrrrrrrrrrrr} 
    \toprule
    && \multicolumn{2}{c}{5\% projection} && \multicolumn{2}{c}{10\% projection} \\
    \cmidrule{3-4}\cmidrule{6-7}
    \rule{0pt}{10pt} 
        Sparsity && Time & Quality && Time & Quality \\
        \midrule
        99 &&    - & - &&        - &       -  \\
        98 &&    - &  - &&    27.80 &    96.64 \\
        97 &&    - &   - &&    39.29 &    97.04 \\
        96 &&    - & - &&    79.11 &    97.13 \\
        95 &&    - & - &&   114.72 &    97.16 \\
        94 &&    - &     - &&    80.08 &    97.18 \\
        93 &&    24.91 &    96.37 &&   101.40 &    97.20 \\
        92 &&    23.57 &    96.50 &&   157.81 &    97.21 \\
    \bottomrule
    \end{tabular}
    \label{tab: sparsity}
\end{table}

It is also interesting to see the results of \textsc{max-2-sat} since even though the feasible region is the same (meaning that the same feasibility argument applies), the objective matrix is different. 
Results for randomly generated formulas generated from a DIMACS\footnote{\href{https://dimacs.rutgers.edu/pub/challenge/satisfiability/cotributed/selman/}{https://dimacs.rutgers.edu/pub/challenge/satisfiability/cotributed/selman/}} challenge are presented in Table \ref{tab: maxsat}. 
We can see how we can get an approximation of around 70\% for most cases by projecting to either 10\% or 20\% of the original size. 

    \begin{table}[!htbp]
        \centering
        \captionof{table}{Relative time (\%) and relative quality (\%) with respect to the original relaxation of \textsc{max-2-sat} for formulas with different number of clauses $n$ and clause to variable ratio $C$ with sparse projectors from D'Ambrosio et al. \cite{DAmbrosio2020}}
        \begin{tabular}{lrrrrrrrrrrrrrrr} 
            \toprule
            & & && \multicolumn{2}{c}{10\% projection} && \multicolumn{2}{c}{20\% projection} \\
            \cmidrule{5-6} \cmidrule{8-9}
            \rule{0pt}{10pt} 
            Instance & $n$ & $C$ && Time & Quality && Time & Quality \\
            \midrule
                          urand-5000-5000 & 5000  &      1.0 &&  11.54 & 66.95 &&  45.78 & 75.87  \\
             urand-5000-5100 & 5000  &     1.02 &&  22.03 & 67.28 &&  39.82 & 75.52  \\
             urand-5000-5200 & 5000  &     1.04 &&  15.90 & 67.27 &&  44.07 & 75.62  \\
             urand-5000-5300 & 5000  &     1.06 &&  13.87 & 66.78 &&  44.26 & 75.38  \\
             urand-5000-5500 & 5000  &      1.1 &&  11.35 & 67.09 &&  41.66 & 75.46  \\
             urand-5000-5800 & 5000  &     1.16 &&  16.29 & 66.63 &&  34.65 & 74.97  \\
             urand-5000-5900 & 5000  &     1.18 &&  11.32 & 66.92 &&  40.89 & 75.13  \\
             urand-5000-6000 & 5000  &      1.2 &&  10.62 & 66.25 &&  44.20 & 75.10  \\
             urand-5000-6100 & 5000  &     1.22 &&  20.97 & 66.71 &&  39.44 & 74.89  \\
             urand-5000-6200 & 5000  &     1.24 &&  19.39 & 66.81 &&  38.51 & 75.31  \\
             urand-5000-6300 & 5000  &     1.26 &&  14.68 & 67.02 &&  44.66 & 75.21  \\
             urand-5000-6400 & 5000  &     1.28 &&  18.34 & 66.80 &&  41.28 & 74.93  \\
             urand-5000-6500 & 5000  &      1.3 &&  19.74 & 66.63 &&  43.44 & 74.55  \\
             urand-5000-6600 & 5000  &     1.32 &&  13.49 & 66.49 &&  46.54 & 74.74  \\
             urand-5000-6700 & 5000  &     1.34 &&  18.65 & 66.64 &&  46.45 & 74.88  \\
             urand-5000-6800 & 5000  &     1.36 &&  14.86 & 66.61 &&  44.34 & 75.08  \\
             urand-5000-6900 & 5000  &     1.38 &&  12.67 & 66.59 &&  45.53 & 74.55  \\
             urand-5000-7100 & 5000  &     1.42 &&  14.96 & 66.55 &&  41.86 & 74.39  \\
            \bottomrule
        \end{tabular}
        \label{tab: maxsat}
    \end{table}

Now, we are interested to see if the quality of the projections becomes better as we increase the dimension of the graphs. Following this, Figure \ref{fig: quality} shows the average (among 5 graphs per size) progression of the quality as we increase the number of vertices of graphs generated by rudy with a density of 30\%. Here we can see that we get better and better approximations by projecting the matrix to a tenth of its original size. Moreover, the computation savings also grow with dimension, reaching a point in which the original problem cannot be solved but the projected one can. 

\begin{figure}
    \centering
    \includegraphics[width=0.6\linewidth]{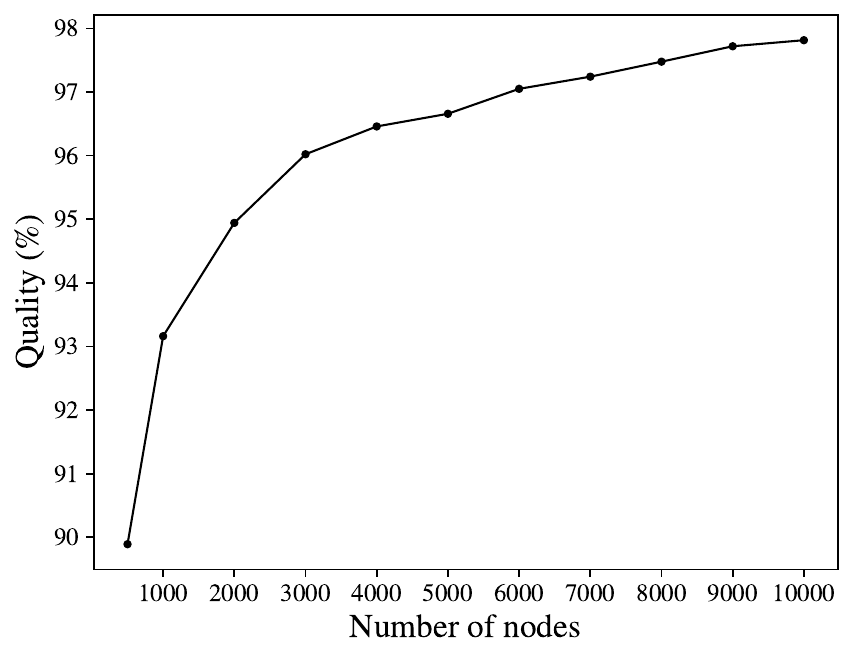}
    \caption{Geometric mean of approximation for a 10\% projection of \textsc{maxcut} for different graph sizes of rudy graphs with 30\% density}
    \label{fig: quality}
\end{figure}

Lastly, Lemma \ref{lemma: approx feasibility} can also be used to solve SDP feasibility problems of reduced dimension. By correctly picking the projection size and the projector density we can ensure that the projected feasibility problem will be feasible with high probability given that the original problem is feasible. If the original problem is infeasible, the projected problem is infeasible with probability 1 since the feasible region of the projected problem is always contained in the feasible region of the original problem, as indicated in Corollary \ref{cor: inclusion of feasible regions}. Now, a possible application of this is solving the Gap Relaxation of the satisfiability problem \textsc{2-sat}, or other SDP relaxations of more general $k$-\textsc{sat} problems \cite{deKlerk2000, deKlerk2003}. The SDP relaxation of \textsc{2-sat} is
\begin{align} \label{gap relaxation}
\text{find} \ & X \succeq 0\\ \notag
\text{s.t.} \ & \text{diag}(X) = 1\\
& \text{sign}(x_i) X_{0,i} + \text{sign}(x_j) X_{0,j} - \text{sign}(x_ix_j)X_{0,j} =  1 \quad \forall(i,j) \in C, \notag
\end{align} 
where $C$ is the set of clauses involving two variables, and entry $X_{0,i}$ is associated with $x_i$ and $X_{i,j}$ with $x_i x_j$. We can solve this problem to detect unsatisfiability. If this relaxation is not feasible, then the formula is not satisfiable. If the projected problem is not feasible, then we know for sure that the gap relaxation \eqref{gap relaxation} is also infeasible. However, the projected problem may be infeasible while the original one is feasible. 
But if the projected dimension and the sparsity of the projector are chosen right, this happens with very low probability. 
Table \ref{tab: Gap relaxation} shows how for 20\% and 50\% projections, we can in many cases trust the result of the projection. 

\begin{table}[!htbp]
    \centering
    \captionof{table}{Proportion (\%) of feasible (satisfiable) instances of \textsc{2-sat} that are also feasible after projection (among 10 satisfiable randomly generated formulas).}
    \begin{tabular}{lrrrrrrrrrrrrrrr} 
        \toprule
        n & C & 20\% projection & 50\% projection \\
        \midrule
                  500 &      0.5 & 100  & 100  \\
                      &        1 & 0  & 100  \\
                      &        2 & 0  & 100  \\
             \midrule
                 1000 &      0.5 & 100 & 100  \\
                      &        1 & 0 & 100 \\
                      &        2 & 0  & 100  \\
             \midrule
                 2000 &      0.5 & 100  & 100 \\
                      &        1 & 0  & 100 \\
                      &        2 & 0  & 100 \\
            \bottomrule
    \end{tabular}
    \label{tab: Gap relaxation}
\end{table}

\subsection{Binary Polynomial Optimization}
Polynomial optimization refers to the study of optimization problems where the objective and the constraint functions are polynomials. 
Many optimization problems fall into this category, including binary problems since one can encode binary variables in polynomial constraints.
Problems with binary variables can be written as polynomial optimization problems by adding the constraint $x^2 = x$. 
However, polynomial optimization problems become nonlinear as soon as the degree of the polynomial is larger than 1, and one does not have to increase the degree too much before these problems become extremely difficult to solve. 
For instance, polynomial optimization problems with objective function and constraints of degree 2 are quadratically constrained quadratic programs (QCQP), which are known to be NP-hard for the general case \cite{boyd_vandenberghe_2004}. 
Moreover, Nesterov showed in \cite{Nesterov2003RandomWI} that the maximization of cubic or quartic forms over the Euclidean ball is NP-hard. 
But still, polynomial optimization problems appear in many applications, such as graph and combinatorics problems \cite{Gvozdenovi2006, Laurent2008, Laurent2003}, or to search for Lyapunov functions in control problems \cite{Parrilo2008, Didier2005}. 
In the early 2000s, a new theory was developed that allowed to approximate any polynomial optimization problem by solving a sequence of convex optimization problems. 
In particular, Lasserre \cite{Lasserre2001} and Parrilo \cite{Parrilo2003} showed that we can construct a hierarchy of semidefinite programming (SDP) problems that converges to the optimal global solution of the polynomial optimization problem. 
Even though SDP problems can be solved to any desired accuracy in polynomial time using algorithms such as interior method (IP) methods \cite{Nesterov1994}, they scale quite badly in practice since current technology cannot handle the memory requirements of IP methods.
This means that in practice, since the size of the aforementioned hierarchies grows very quickly, for most practical applications one cannot solve more than the first level of problems of up to 10 variables \cite{Ahmadi2015}. 

We now present experiments for the SOS relaxation of the Stable Set Problem. 
But first, let us quickly introduce SOS relaxations of PO problems. 
consider the following polynomial optimization problem
\begin{align*}\tag{PO}\label{polynomial optimization problem}
    \min \quad & f(x) \\
    \text{s.t.} \quad & g_i(x) \geq 0 \quad i=1, \hdots, s
\end{align*}
where $f(x)$ and $g_i(x)$ for all $i=1, \hdots, s$ are polynomials. 
Note that finding the optimal value of this function is equal to finding the largest $\gamma$ so that when we subtract $\gamma$ from the objective function we remain nonnegative for all feasible $x$. 
Following this, the optimal value of \eqref{polynomial optimization problem} is equivalent to that of the following problem
\begin{align*}
    \max \quad & \gamma \\
    \text{s.t.} \quad & f(x) - \gamma \geq 0 \quad \forall x \in K
\end{align*}
where K is the semialgebraic set $K = \{x \mid g_i(x) \geq 0 \text{ for } i=1, \hdots, s\}$. 
In general, characterising the non-negativity of a polynomial over a semialgebraic set is a hard problem.
Under an assumption on $K$ called the archimedean condition, which is similar to compactness (see \cite{Laurent2008} for more details), one can use Putinar's Positivstellensatz to replace the non-negativity condition $f(x) - \gamma \geq 0$ by the following positivity certificate
\begin{align*}\tag{SOS}\label{SOS problem}
    \max \quad & \gamma \\
    \text{s.t.} \quad & f(x) - \gamma = \sigma_0(x) + \sum_{i=1}^s \sigma_i(x)g_i(x), \\
    &\sigma_0, \sigma_i(x) \in \Sigma[x] \text{ for }i=1,\hdots, s 
\end{align*}
where $\Sigma[x]$ is the subring of sum-of-squares polynomials. 
We can then bound the degree of the polynomials $\sigma_0, \sigma_i(x)$ by $t$ to get a hierarchy of approximations (SOS-t) of the original problem \cite{Lasserre2001, Parrilo2003}. 
This hierarchy converges to the optimal solution as we increase the degree bound $t$.
We can then use the fact that a polynomial $p$ of degree $2d$ is a SOS polynomial if and only if it can be written as 
\begin{equation*} 
    p(x) = z(x)^{\top}Qz(x)
\end{equation*}
for some positive semidefinite matrix $Q$ and the vector $z(x)$ of monomials of degree up to $d$. 
Following this, we rewrite the previous (bounded degree) problem as an SDP problem by adding a psd matrix for each SOS polynomial in (SOS-t).
The size of these matrices will be $\binom{n+d}{d} \times \binom{n+d}{d}$ where $2d$ is the degree of the SOS polynomials we want to represent, and $n$ is the number of variables of the problem. 

The stable set problem is the problem of finding the maximum stable set in a graph, where a set of vertices is stable if no vertices are connected with an edge. 
Let $G=(E, V)$ be a non-directed graph defined by the set of vertices $E$, and edges $V$. 
We want to find the maximum number of vertices $v \in V$ such that no two chosen vertices are connected by an edge. 
In the polynomial optimization problem, we encode the binary variables using the polynomial $x_v^2 = x_v$ for each vertex. An advantage of the polynomial formulation is that we can use the Lasserre Hierarchy to obtain tighter bounds than for the IP formulation. 

For some classes of graphs, such as perfect graphs, the first level of the relaxation is enough to get the optimal independence number. 
However, for so-called imperfect graphs, one needs to go to the second, or even further, level to get the optimal solution. 
We are interested in this class of graphs and show that our methodology can be applied to them to reduce the computational burden of solving the second level of the hierarchy.
We focus on four sets of imperfect graphs: generalised Petersen graphs, helm graphs, Jahagir's graph, and imperfect graphs generated with the method introduced in \cite{Dobre2015}, which we refer to as cordones graphs. 
We decided to work with the complement since the resulting SDPs will have fewer constraints, which is a desirable feature for our method. 
We denote generalised Petersen graphs as \textit{petersen-v-k}, where $v$ is the number of vertices in the inner circle, and $k$ is the modulo of the subscripts of the edges. 
Similarly, we denote helm graphs as \textit{helm-v} where $v$ is the number of vertices in the inner and outer circle, and Jahagir's graph as \textit{jahagir-v-k} where $v$ is the number of vertices between inner vertices, and $k$ is the number of inner vertices. 
We denote the graphs taken from Dobre and Vera \cite{Dobre2015} as \textit{cordones-n}, and these graphs have $3n+2$ vertices. For all graphs, we add the prefix \textit{c-} to denote the complement of the graph. 
    \begin{table}[!htbp]
    \centering
    \captionof{table}{Relative time (\%) and relative quality (\%) with respect to different projections of the second level of the Stable Set Problem for dense imperfect graphs with $n$ vertices, $m$ edges and $c$ constraints.}
    \begin{tabular}{lrrrrrrrrrrr} 
        \toprule
        & & & & \multicolumn{1}{c}{original 2\textsuperscript{nd} level} && \multicolumn{3}{c}{projected 2\textsuperscript{nd} level} \\
        \cmidrule{5-5} \cmidrule{7-9}
        \rule{0pt}{10pt} 
        Graph & $n$ & $m$ & $c$  & Size && Size & Qlt & Time \\
        \midrule
        c-petersen-20-2 &       40 &      720 &      102 &      821 &&       82 &       99 &     7.63 \\
             c-petersen-30-2 &       60 &     1680 &      152 &     1831 &&      183 &      100 &     2.24 \\
             c-petersen-40-2 &       80 &     3040 &      202 &     3241 &&      324 &      100 &     4.48 \\
             c-cordones-20 &       62 &     1430 &      525 &     1954 &&      195 &       75 &     7.27 \\
             c-cordones-30 &       92 &     3195 &     1085 &     4279 &&      428 &       87 &    20.48 \\
             c-cordones-40 &      122 &     5660 &     1845 &     7504 &&      750 &       89 &    18.33 \\
             c-helm-21 &       43 &      840 &      129 &      947 &&       95 &      100 &    14.55 \\
             c-helm-31 &       63 &     1860 &      189 &     2017 &&      202 &      100 &     2.29 \\
             c-helm-41 &       83 &     3280 &      249 &     3487 &&      349 &      100 &     7.00 \\
             c-jahangir-17-2 &       35 &      544 &       88 &      631 &&       63 &       99 &     5.88 \\
             c-jahangir-19-3 &       58 &     1577 &      136 &     1712 &&      171 &      100 &     1.94 \\
             c-jahangir-41-2 &       83 &     3280 &      208 &     3487 &&      349 &      100 &     4.18 \\
        \bottomrule
    \end{tabular}
    \label{tab: stable set}
\end{table}

It is important to note that for graphs with a large number of vertices, the moment approach is also simplified considerably. This means that if were to solve the problem using this method, the size of the psd matrices of the second level would not be as large as those of Table \ref{tab: stable set}. However, these experiments show that for polynomial optimization problems with a small number of constraints random projections work well. 

\section{Conclusion}
In this paper we have presented updated probabilistic bounds for some of the most used results in the theory of RPs for MPs that depend on the sparsity parameter of a sub-gaussian projection. 
We have then built on these bounds to show that we can approximately solve semidefinite programming problems by applying RPs.
We then evaluated the performance of the proposed method with some well-known applications of semidefinite relaxations. 
In particular, we show that we can project unweighted SDP relaxations of \textsc{maxcut} instances while achieving very high quality solutions.
Slightly weaker results are obtained for relaxations of \textsc{max-2-sat}. 
We also show that random projections can be applied to feasibility problems, such as the Gap relaxation of \textsc{2-sat}. 
Lastly, we also project the SDP appearing in the second level of the Lasserre Hierarchy of the Stable Set Problem. 
We show that for problems with very few constraints, RPs work well. 
Since in practice, we could only apply this methodology for PO problems that have few constraints, it is also sensible to think that the constraint aggregation approach of \cite{Liberti2021} might also be promising for this application. We have started exploring this research direction, but presenting it here is beyond the scope of this paper. 

\section*{Acknowledgement}
This project has been partially funded by the Japan Society for the Promotion of Science (JSPS) through their pre-doctoral fellowship programme. 

\printbibliography

\end{document}